\documentclass{article}

\usepackage{amsmath, amsthm, amssymb, mathrsfs}
\usepackage[all]{xy}
\input xy
\xyoption{all}
\xyoption{2cell}
\xyoption{v2}

\usepackage{hyperref} 
\usepackage{fullpage}

\newcommand{\sSet}{\mathit{sSet}}
\newcommand{\Dec}{\mathrm{Dec}} 
\newcommand{\W}{\overline{W}} 
\newcommand{\last}{\mathrm{last}} 
\newcommand{\first}{\mathrm{first}}
\newcommand{\Wbar}{\overline{W}}
\newcommand{\cosk}{\mathit{cosk}} 

\newcommand{\Set}{\mathit{Set}}
\newcommand{\Top}{\mathit{Top}}
\newcommand{\cC}{\mathscr{C}}
\renewcommand{\cD}{\mathscr{D}}
\newcommand{\cE}{\mathscr{E}} 
\newcommand{\cK}{\mathscr{K}} 
\newcommand{\cU}{\mathscr{U}} 
\newcommand{\bH}{\mathbf{H}} 
\newcommand{\Pre}{\mathit{Pre}} 
\newcommand{\Sh}{\mathit{Sh}}

\newcommand{\Gpd}{\mathit{Gpd}}
\newcommand{\Gp}{\mathit{Gp}} 

\newcommand{\Open}{\mathrm{Open}} 
\newcommand{\op}{\mathrm{op}}
\newcommand{\Map}{\mathrm{Map}} 
\newcommand{\map}{\mathrm{map}}

\newcommand{\Ho}{\mathrm{Ho}} 
\newcommand{\Cov}{\mathrm{Cov}} 
\renewcommand{\hom}{\mathbf{hom}} 
\newcommand{\dslash}{/\negmedspace /}

\catcode`\@=11\def\varholim@#1#2{\mathop{\vtop{\ialign{##\crcr
 \hfil$#1\m@th\operator@font holim$\hfil\crcr
 \noalign{\nointerlineskip\kern\ex@}#2#1\crcr
 \noalign{\nointerlineskip\kern-\ex@}\crcr}}}}
\def\hocolim{\mathpalette\varholim@\rightarrowfill@} 
\def\hoinvlim{\mathpalette\varholim@\leftarrowfill@}
\catcode`\@=\active

\newcommand{\AUT}{\mathrm{AUT}} 
\newcommand{\Aut}{\mathrm{Aut}}

\renewcommand{\a}{\alpha}

\newtheorem{theorem}{Theorem}[section]

\newtheorem{proposition}{Proposition}[section]
\newtheorem{lemma}{Lemma}[section] 
\newtheorem{corollary}{Corollary}[section]

\theoremstyle{definition}
\newtheorem{definition}{Definition}[section]

\begin{document} 

\title{Classifying theory for simplicial parametrized groups} 
\author{Danny Stevenson\thanks{This work was supported by the 
 Engineering and Physical Sciences Research Council [grant number EP/I010610/1] 
 and the Australian Research Council [grant number DP120100106]}\\
School of Mathematical Sciences\\
The University of Adelaide  \\
SA 5005\\
Australia\\
email: {\tt daniel.stevenson@adelaide.edu.au} }

\maketitle 

\begin{abstract}
In this paper we describe a classifying theory for families of simplicial 
topological groups.  If $B$ is a topological space and 
$G$ is a simplicial topological group, then we 
can consider the non-abelian cohomology $H(B,G)$ of $B$ with coefficients in $G$.  
If $G$ is a topological group, thought of as a constant simplicial 
group, then the set $H(B,G)$ is the usual set of isomorphism classes of 
principal $G$-bundles, or $G$-torsors, on $B$.  For more general simplicial groups $G$, 
the set $H(B,G)$ parametrizes the set of equivalence classes of 
higher $G$-torsors.  In this paper we consider a more general setting where 
$G$ is replaced by a simplicial group in the category of spaces over $B$.  
The main result of the paper is that under suitable conditions on $B$ 
and $G$ there is an isomorphism between $H(B,G)$ and the set 
of isomorphism classes of parametrized principal bundles on $B$, 
with structure group given by the fiberwise geometric realization of $G$.      
 
\end{abstract}

\tableofcontents

\section{Introduction} 

Let $G$ be a presheaf of groups on a topological space $B$.  Then the \v{C}ech cohomology 
$\check{H}^1(B,G)$ of $B$ with coefficients in $G$ is traditionally defined in terms of $G$-valued cocycles $g_{ij}\in G(U_i\cap U_j)$ 
relative to some open cover $(U_i)_{i\in I}$ of $B$.  The cocycle condition satisfied by the $g_{ij}$ is the equation 
\[ 
g_{ij}g_{jk} = g_{ik} 
 \]
in $G(U_i\cap U_j\cap U_k)$.  The set $\check{H}^1(B,G)$ 
of equivalence classes of such cocycles parametrizes isomorphism classes of $G$ 
{\em torsors} on $B$, i.e.\ presheaves on $B$ equipped with a principal action of $G$.  
The typical case is when $G$ is the sheaf of groups on $B$ represented by a 
topological group $G$, in which case $\check{H}^1(B,G)$ parametrizes the set of principal $G$-bundles on $B$.  
Slightly more generally, we could consider a parametrized situation    
where the group $G$ is replaced by a family of topological groups $G_b$  
labelled by the points $b$ of $B$.  In other words, we consider $G$ as a 
group object in the category of spaces over $B$.  Then $G$ determines a sheaf on 
$B$, which we also denote by $G$, whose sections are the local sections of the 
projection to $B$, and the set $\check{H}^1(B,G)$ 
parametrizes the set of isomorphism classes of {\em parametrized} principal $G$-bundles on $B$ (see \cite{CJ,RS}).    

Now suppose that for any open set $U\subset B$, 
the group of sections $G(U)$ is replaced by a {\em simplicial} group 
of sections $G(U)$.  In other words, 
suppose that $G$ is a group object in the category of simplicial presheaves on $B$.  
There is a generalization 
of the \v{C}ech cohomology 
set above to a cohomology set $\check{H}(B,G)$.  
Analogously there is a notion of `higher $G$-torsor' on 
$B$ and under suitable hypotheses $\check{H}(B,G)$ parametrizes the set of 
isomorphism classes of these higher torsors.  
While higher torsors are implicitly the subject matter of this paper, we 
will have no need to consider them explicitly (apart 
from a brief discussion in Section~\ref{non abelian cohomology}) and 
we instead refer the 
interested reader to \cite{Baez-St, Bartels, Breen1, Breen2, JL, L-GSX, Wockel} among others for further details.  

A cocycle 
in this generalized \v{C}ech cohomology is again defined 
relative to some open cover $U = (U_i)_{i\in I}$ of $B$, but 
now the data defining the cocycle is more intricate (an 
early reference where these cocycles are made explicit in the 
context of classifying fibrations is the thesis \cite{Wirth} of Wirth).  
To begin with we have sections $g_{ij}\in G_0(U_i\cap U_j)$ of 
0-simplices, but now the equation $g_{ij}g_{jk} = g_{ik}$ above 
will only be satisfied up to homotopy --- thus 
there will be sections $g_{ijk}\in G_1(U_i\cap U_j\cap U_k)$ 
which satisfy the equations 
\[
d_0(g_{ijk}) = g_{ik},\quad d_1(g_{ijk}) = g_{ij}g_{jk}
\]
in $G_0(U_i\cap U_j\cap U_k)$.  In turn these sections $g_{ijk}$ will 
satisfy a cocycle equation up to homotopy over quadruple intersections and so on.  

It is much more illuminating to give this a more 
conceptual reformulation and think of the cocycle 
as a morphism of {\em simplicial presheaves} on $B$.  
Furthermore, working in the category $s\Pre(B)$ of 
simplicial presheaves on $B$ allows us to take advantage 
of the model category structure on 
$s\Pre(B)$ constructed in \cite{Jardine1} so that we 
have powerful homotopy theoretic tools at our disposal.  

Roughly speaking, this model structure gives rise to a non-abelian version 
of the derived category construction.  If $X$ is a simplicial presheaf on $B$, 
then for any point $p\in B$ one can define a simplicial set $X_p$, called the stalk of $X$ at $p$.  
A map of simplicial presheaves $X\to Y$ is said to be a weak equivalence if it induces  
a weak homotopy equivalence $X_p\to Y_p$ on stalks for all points $p$ of $B$.  
This notion of weak equivalence is a generalization of the notion of quasi-isomorphism 
of complexes of sheaves: if $X$ and $Y$ are presheaves of simplicial abelian groups 
then a map $X\to Y$ is a weak equivalence precisely if it is a quasi-isomorphism of the 
corresponding complexes of sheaves. 

In analogy with the derived category construction we can consider the homotopy category 
$\Ho(s\Pre(B))$ with the same objects as $s\Pre(B)$ but in which 
these weak equivalences are now isomorphisms.  The fact 
that these weak equivalences in $s\Pre(B)$ form part of the structure of a model category 
means that there is a simple description of the set of morphisms $[X,Y]$ in 
$\Ho(s\Pre(B))$.  It turns out that when $G$ is a presheaf of groups on $B$ as above, 
then there is a bijective correspondence between the \v{C}ech cohomology set $\check{H}^1(B,G)$ 
and $[1,\Wbar G]$, where $1$ denotes the terminal simplicial presheaf on $B$ and $\Wbar G$ is 
the classifying simplicial presheaf of the group $G$.  

More generally, if $A$ is a sheaf of abelian groups on $B$, then the sheaf cohomology 
group $H^n(B,A)$ is isomorphic to $[1,K(A,n)]$, where $K(A,n)$ is a simplicial sheaf on $B$ 
whose stalk at a point $p\in B$ is the Eilenberg-Mac Lane space $K(A_p,n)$.  In general there is a 
subtle difference between \v{C}ech cohomology and sheaf cohomology; this distinction 
forces us to work with so-called `hyper-\v{C}ech cohomology' in which ordinary covers 
of $B$ by open sets are replaced with `hypercovers'.  If $G$ is a presheaf of 
simplicial groups on $B$, then we will define $H(B,G) = [1,\Wbar G]$, where $\Wbar G$ 
denotes the classifying simplicial presheaf of the simplicial group $G$.  Elements of 
$H(B,G)$ are represented by $G$-cocycles on $B$ in a sense that we will make 
precise in Section~\ref{non abelian cohomology}.  

In this paper we want to study $H(B,G)$ in the following special case.  Let $\cK$ 
denote the category of $k$-spaces \cite{Vogt} and suppose 
that $B$ is a paracompact, Hausdorff space.  Write $\cK_{/B}$ for 
the category of spaces over $B$.  We consider a {\em parametrized simplicial group} $G$, i.e.\ 
a simplicial group in $\cK_{/B}$,  
so that $G$ is a simplicial $k$-space, augmented over $B$, which is 
equipped with a simplicial map $G\times_B G\to G$ giving 
$G_n$ the structure of a fiberwise group over $B$ for all $n\geq 0$.    
Then $G$ determines a presheaf (in fact a sheaf) of simplicial groups on $B$, which 
we will also denote by $G$.  The simplicial set of sections of $G$ over an open 
set $U\subset B$ has as its set of $n$-simplices the set of sections of the map $G_n\to B$ 
over $U$.  
 
By taking the fiberwise geometric realization of the simplicial group $G$ in $\cK_{/B}$ 
we obtain a group $|G|$ in $\cK_{/B}$.  We can therefore form the ordinary \v{C}ech cohomology 
$\check{H}^1(B,|G|)$ as described above.  To state the main result of our paper 
we must also assume that the simplicial parametrized group $G$ is {\em fibrant} in the sense that the structure 
map $G\to B$ is an $f$-fibration (see Theorem~\ref{may sig model thm} below).  
We then have

\begin{theorem} 
\label{main result}
Let $B$ be a paracompact, Hausdorff space and let $G$ 
be a fibrant simplicial parametrized group in $\cK_{/B}$.  
If $G_n$ is well sectioned for all $n\geq 0$ then there is an isomorphism 
of sets 
\[ 
H(B,G) = \check{H}^1(B,|G|).  
\]
\end{theorem}

We discuss some consequences of this theorem.  To begin with, the parametrized 
setting is a natural topological environment for the sheaf theoretic 
notions of higher torsors discussed in \cite{Breen1, Breen2}.  However, Theorem~\ref{main result} is 
interesting (and new) even if the parametrized group $G$ is trivial in the sense that 
it is equal to the product $B\times G$ for some simplicial group $G$ in $\cK$.  
In this case it follows from results of \cite{RS} that the isomorphism 
of    
Theorem~\ref{main result} can 
be restated as an isomorphism of sets
\begin{equation}
\label{main iso}
H(B,G) = [B,B|G|],   
\end{equation}
where $[B,B|G|]$ denotes the set of homotopy classes of maps from $B$ into 
$B|G|$.  Thus $B|G|$ is a classifying space for non-abelian cohomology 
with coefficients in $G$.  
In fact our proof of Theorem~\ref{main result} shows more than what is stated in this case: in 
the course of our proof we construct a {\em universal} higher $G$-torsor on 
$B|G|$ with the property that every higher $G$-torsor on 
$B$ is equivalent to one obtained from this universal higher torsor via a map from $B$ to $B|G|$, 
unique up to homotopy of maps.  

In another direction, if $G$ is a simplicial abelian group in $\cK_{/B}$, then $\Wbar G$ is also 
a simplicial abelian group in $\cK_{/B}$ and it follows that $H(B,G)$ has a natural structure as an abelian group.  
Likewise, the fiberwise geometric realization $|G|$ has the structure of an abelian group in 
$\cK_{/B}$ and hence $\check{H}^1(B,|G|)$ also has a natural structure of an abelian group.  
Another by-product of our proof of Theorem~\ref{main result} is that, under the 
above hypotheses, the isomorphism~\eqref{main iso} 
above is an isomorphism of abelian groups in this case (see Section~\ref{sec:univ cocycle}).  
In particular, if $A$ is a well pointed abelian group in $\cK$, then (by the usual 
mechanism of the Dold-Kan correspondence and using the results of \cite{Brown}) there is a simplicial abelian 
group $A[n]$ in $\cK$ for any integer $n\geq 0$ with the property that 
$H(B,A[n])$ is isomorphic to the \v{C}ech cohomology group $\check{H}^{n+1}(B,\underline{A})$ 
(at least under the assumption that $B$ is paracompact), 
where $\underline{A}$ denotes the sheaf of abelian groups on $B$ whose sections 
over an open set $U\subset B$ are the continuous maps from $U$ into $A$.  In this case 
Theorem~\ref{main result} can be interpreted as an isomorphism between $\check{H}^{n+1}(B,\underline{A})$ 
and $[B,B|A[n]|]$.  In the special case that $A$ is discrete, we recover the result of 
Huber \cite{Huber} giving an identification of $\check{H}^{n+1}(B,A)$ with $[B,K(A,n+1)]$.

We need to comment on the relationship of this paper with previous works of other authors.  
We are not aware of any analogous studies in the parametrized setting.  In the case where the 
simplicial group object $G$ is trivial, in the sense that it takes the form of a trivial bundle 
of groups $B\times G\to B$ for a simplicial topological group $G$, then the most general result related to 
ours that we are aware of is contained in the paper \cite{Segal2} of Graeme Segal.  In this paper 
Segal proves (Proposition 4.3 of \cite{Segal2}) that if $B$ is a 
paracompact space and $A$ is a good simplicial space then 
there is a bijection between $[B,|A|]$ and the set of {\em concordance classes} 
of $A$-bundles on $B$.  In \cite{BBK} a related version of Segal's 
theorem is proven, in which the set of homotopy classes 
$[B,|A|]$ is shown to be isomorphic to the set of concordance classes of 
$A$ valued cocycles on $B$, where $A$ is again a good simplicial space and 
$B$ is a space with the homotopy type of a CW complex.  
We caution the reader that there is an {\em a priori} difference between the concordance relation 
and the relation on simplicial maps given by simplicial homotopy.  For example, 
it is not true in general that the set of isomorphism classes of principal groupoid 
bundles is in a bijective correspondence with the set of concordance classes of 
principal groupoid bundles (see for instance \cite{MRS}).  It is of course 
well known that there is a bijection between the set of 
isomorphism classes and concordance classes of ordinary principal bundles with 
structure group a topological group.  
One outcome of our work 
is that the simplicial homotopy relation and the concordance relation agree when $A = \Wbar G$ 
for a simplicial topological group $G$ (under the conditions on $G$ described in Theorem~\ref{main result}).  

In another direction, again in the case of the trivial bundle of groups $G\times B\to B$, 
there is the work \cite{BGPT} on crossed complexes.  In this paper the authors 
show that there is a bijective correspondence between the set of homotopy classes of maps $[M,BC]$ for 
$M$ a filtered space and $C$ a crossed complex, and the set of equivalence classes of $C$-valued cocycles 
on $M$.  The equivalence relation considered in \cite{BGPT} arises from the structure of 
$M$ as a filtered space and it is not clear how this relates to the simplicial homotopy relation on 
simplicial maps.

Our proof of Theorem~\ref{main result} is a substantial generalization of the methods of \cite{Baez-St}.  
We relate the non abelian cohomology set $H(B,G)$ to the ordinary non abelian cohomology $\check{H}^1(B,|G|)$ by constructing 
a simplicial group $|\Dec\, G|$ in $\cK_{/B}$ together with homomorphisms $|G| \leftarrow |\Dec\, G|\to G$ 
which give us a means to compare $G$ and $|G|$ through the intermediate group $|\Dec\, G|$.  
The simplicial group $|\Dec\, G|$ is the fiberwise geometric realization of the bisimplicial group 
$\Dec\, G$ in $\cK_{/B}$ constructed from $G$ using Illusie's total d\'{e}calage functor \cite{Illusie2}.  The bisimplicial 
group $\Dec\, G$ also plays a prominent role in Porter's work \cite{Porter} on $n$-types.  By applying 
the classifying space functor $\Wbar$ to the diagram of groups above, we obtain a 
diagram 
\[
\xymatrix{ 
\Wbar |\Dec\, G| \ar[d] \ar[r] & \Wbar G \\ 
\Wbar |G| } 
\]
in $s\cK_{/B}$.  Our main technical result (Proposition~\ref{hypercover}) is that the map $\Wbar |\Dec\, G|\to \Wbar |G|$ 
is a hypercover; this in turn relies partly on a result from \cite{RS} (recalled as Theorem~\ref{main thm from RS} below).  
The diagram above should be regarded as a `universal $G$-torsor over $\Wbar |G|$';  we show  
that this diagram induces isomorphisms $H(B,|\Dec\, G|)\cong \check{H}^1(B,|G|)$ and $H(B,|\Dec\, G|)\cong H(B,G)$.  

In summary, the contents of the paper are as follows.  
In Section~\ref{sec:background} we describe some background material on parametrized spaces 
from \cite{May-Sig}, some simplicial techniques based around the d\'{e}calage comonad, 
and finally some material on generalized matching objects.  
In Section~\ref{sec:internal} we recall some background on the 
homotopy theory of simplicial sheaves and simplicial presheaves, 
and also study the notion of a locally fibrant 
simplicial object in a category equipped with a Grothendieck topology introduced in \cite{Henriques}.  
We also study the `internal' local simplicial homotopy theory corresponding 
to this notion of locally fibrant simplicial object.  A key result in this section is 
Proposition~\ref{contractibility propn} where we give a criterion to detect 
when certain homomorphisms of group objects are locally acyclic local fibrations.  
In Section~\ref{sec:nonabelian} we prove some results about the internal local 
simplicial homotopy theory developed in the previous section, and give a 
precise definition of the non-abelian cohomology sets $H(B,G)$.  Finally in 
Section~\ref{sec:proof of main thm} we give our proof of Theorem~\ref{main result}.

\section{Background} 
\label{sec:background} 

\subsection{Parametrized spaces} 

Let $\cK$ denote the category of $k$-spaces \cite{Vogt} 
and let $\cU$ denote the subcategory of compactly generated spaces (i.e.\ weakly Hausdorff $k$-spaces).  
We will be interested in the category $\cK_{/B}$ of spaces over $B$, where $B$ is an object of $\cU$.   
Recall from \cite{May-Sig} that $\cK_{/B}$ is a {\em topological bicomplete category}, 
in the sense that $\cK_{/B}$ is enriched over $\cK$, the underlying category is complete and 
cocomplete, and that it is tensored and cotensored over 
$\cK$.  For any space $K$ and space $X$ over $B$ the tensor $K\otimes X$ is 
defined to be the space $K\times X$ in $\cK$, considered as a space over $B$ via the obvious map $K\times X\to B$.  
Similarly, the cotensor $X^K$ is defined to be the space $\Map_B(K,X)$ given by the pullback diagram 
\[ 
\xymatrix{ 
\Map_B(K,X) \ar[r] \ar[d] & \Map(K,X) \ar[d] \\ 
B \ar[r] & \Map(K,B) } 
\] 
in $\cK$, where the map $B\to \Map(K,B)$ is the adjoint of $B\times K\to B$.  
Recall also (see \cite{May-Sig}) that $\cK_{/B}$ is cartesian closed  
under the fiberwise cartesian product $X\times_B Y$ and the fiberwise mapping space $\Map_B(X,Y)$ over $B$.  
The definition of the fiberwise mapping space $\Map_B(X,Y)$ 
is rather subtle and we will not give it here, we instead refer the reader to Definition~1.3.7 of \cite{May-Sig}. 
Let us note though that $\Map_B(X,Y)$ is generally not weak Hausdorff even if 
$X$ and $Y$ are, which is one of the main reasons why May and Sigurdsson choose to 
work with the category $\cK_{/B}$ rather than the category $\cU/B$.

In \cite{May-Sig} several model structures on $\cK_{/B}$ are introduced.  We shall be 
interested in the 
$f$-model structure (for fiberwise) with the weak equivalences, 
fibrations and cofibrations labelled accordingly.  
 Thus a map $g\colon X\to Y$ in $\cK_{/B}$ is called an $f$-{\em equivalence} 
 if it is a fiberwise homotopy equivalence.  This needs the notion of 
 homotopy over $B$, which is formulated in terms of 
$X\times_B I$.  
A map $g\colon X\to Y$ in $\cK_{/B}$ is called an $f$-{\em fibration} 
if it has the fiberwise covering homotopy property, i.e.\ if it has the 
right lifting property (RLP) with respect to all maps of the form 
$i_0\colon Z\to Z\times_B I$ for all $Z\in \cK_{/B}$.  A map $g\colon X\to Y$ 
in $\cK_{/B}$ is called an $\bar{f}$-{\em cofibration}  
if it has the left lifting property (LLP) with respect to all $f$-acyclic $f$-fibrations.  
We have the following result from \cite{May-Sig}.   
\begin{theorem}[May-Sigurdsson] 
\label{may sig model thm}
$\cK_{/B}$ has the structure of a proper, topological model category for which 
\begin{itemize} 
\item the weak equivalences are the $f$-equivalences, 

\item the fibrations are the $f$-fibrations, 

\item the cofibrations are the $\bar{f}$-cofibrations.   
\end{itemize} 
\end{theorem} 

Recall that a model category is said to be {\em topological} if it 
is a $\cK$-model category in the sense of Definition 4.2.18 of \cite{Hovey}, 
for the monoidal model structure on $\cK$ given by the classical 
Str\o m model structure \cite{Strom} on $\cK$.  

In addition to the notion of $\bar{f}$-cofibration, 
there is also the notion of an $f$-{\em cofibration}: this is a map 
$g\colon X\to Y$ which satisfies the LLP with respect to all maps of the form 
$p_0\colon \Map_B(I,Z)\to Z$ for some $Z\in \cK_{/B}$.  
Every $\bar{f}$-cofibration $g\colon X\to Y$ in $\cK_{/B}$ is an 
$f$-cofibration.  The converse is not true in general.  However May 
and Sigurdsson prove (see Theorems 4.4.4 and 5.2.8 of \cite{May-Sig}) 
that if $g\colon X\to Y$ is a closed $f$-cofibration then $g$ is an $\bar{f}$-cofibration.

Moreover, in analogy with the standard characterization of 
closed Hurewicz cofibrations in terms of NDR pairs, May and 
Sigurdsson give a criterion (see Lemma 5.2.4 of \cite{May-Sig}) 
which detects when a closed inclusion $i\colon A\to X$ in 
$\cK_{/B}$ is an $\bar{f}$-cofibration.  
Such an inclusion is an 
$\bar{f}$-cofibration if and only if $(X,A)$ is a 
{\em fiberwise NDR pair} in the sense that there is a map 
$u\colon X\to I$ for which $A = u^{-1}(0)$ and a homotopy $h\colon X\times_B I\to X$ 
over $B$ such that $h_0 = \mathit{id}$, $h_t|_A = \mathit{id}_A$ 
for all $0\leq t\leq 1$ and $h_1(x)\in A$ whenever $u(x) < 1$

The analogue of a pointed space in this parametrized context is the notion of an {\em ex-space}.  An ex-space over 
$B$ is a space $X$ in $\cK_{/B}$ together with a section of the structure map $X\to B$.  The category of ex-spaces and maps 
between them is denoted by $\cK_B$.  An ex-space $X$ is said to be {\em well-sectioned} if 
the section $B\to X$ is an $\bar{f}$-cofibration.  A prime example of an ex-space is a parametrized 
group, i.e.\ a group object in $\cK_{/B}$ --- in this case the identity section gives the required 
section of the structure map.

\subsection{Some simplicial techniques} 

Recall (see for example VII Section 5 of \cite{CWM}) that the 
augmented simplex category $\Delta_a$ is a monoidal category under the operation 
of {\em ordinal sum} denoted here by $\sigma([m], [n])$, with unit given by 
the empty set $[-1]$.  Moreover $[0]$ is a monoid in $\Delta_a$, with multiplication 
given by the unique map $[1]\to [0]$,  which is universal in a certain precise sense (see VII Section 5 Proposition 1 
of \cite{CWM}).  It follows that $[0]$ determines a comonoid in $\Delta_a^\op$ and hence a comonad 
\[
\sigma(-, [0])\colon \Delta^\op\to \Delta^\op
\]
on the opposite category of the simplex category $\Delta$.  This comonad induces through composition a comonad 
on the functor category $s\cC = [\Delta^\op,\cC]$ of simplicial objects in $\cC$, for any category $\cC$.  This comonad is 
denoted 
\[
\Dec_0\colon s\cC\to s\cC
\]
and if $X$ is an object of $s\cC$, then $\Dec_0X$ is called the {\em d\'{e}calage} of $X$.  
It is not hard to see that 
$\Dec_0X$ is an augmented simplicial object whose object of $n$-simplices is 
\[ 
(\Dec_0X)_n = X_{n+1},  
\] 
whose face maps $d_i\colon (\Dec_0 X)_n\to 
(\Dec_0 X)_{n-1}$ are given by $d_i\colon X_{n+1}\to X_n$ for 
$i=0,1,\ldots, n$, and whose degeneracy maps 
$s_i\colon (\Dec_0 X)_n\to (\Dec_0 X)_{n+1}$ are given by 
$s_i\colon X_{n+1}\to X_{n+2}$ for $i=0,1,\ldots, n$.  
The augmentation $(\Dec_0 X)_0\to X_0$ is given 
by $d_0\colon X_1\to X_0$.  We can picture $\Dec_0X$ as follows:
\[ 
\xymatrix{ 
X_0 & \ar[l]_-{d_0} X_1 \ar@<-2.5ex>[r]^-{s_0} & X_2 
\ar@<-2.5ex>[l]_-{d_0} \ar[l]_-{d_1} \ar@<-2.5ex>[r]^-{s_0} \ar@<-5ex>[r]^-{s_1} 
& X_3\  \ar@<-5ex>[l]_-{d_0} \ar@<-2.5ex>[l]_-{d_1} \ar[l]_-{d_2}  \cdots } 
\]   
Thus $\Dec_0X$ is obtained from $X$ by `stripping off' the last face and degeneracy map at each 
level and re-indexing by shifting degrees up by one.  If $\cC$ is 
cocomplete then the leftover degeneracy map at each level can 
be used to define a contraction of $\Dec_0X$ onto $X_0$.  More precisely there is a simplicial homotopy 
$h\colon \Dec_0 X\otimes \Delta^1\to \Dec_0X$ which fits into a commutative diagram 
\[
\xymatrix{ 
\Dec_0 X \ar[d] \ar[dr]^-{s_0d_0} & \\ 
\Dec_0 X\otimes \Delta^1 \ar[r]^-h & \Dec_0 X \\ 
\Dec_0 X \ar[u] \ar[ur]_-{\mathrm{id}} } 
\]
Here the tensor $\Dec_0 X\otimes \Delta^1$ is defined with respect to the usual simplicial structure on 
$s\cC$: thus if $A$ is an object of $s\cC$ and $K$ is a simplicial set then $A\otimes K$ is the simplicial object 
of $\cC$ whose object of $n$-simplices is   
\[
(A\otimes K)_n = \coprod_{k\in K_n} A_n, 
\]
see for example page 85 of \cite{GJ}.  As mentioned above, the functor $\Dec_0$ is a comonad on $s\cC$, whose counit 
is the simplicial map $\Dec_0 X\to X$ given in degree $n$ by the last face map, i.e.\ 
the map  
\[
d_{n+1}\colon X_{n+1}\to X_n.  
\]
We will write $d_\last\colon \Dec_0 X\to X$ for this map.    
Observe that since the functor $\Dec_0\colon s\cC\to s\cC$ is obtained by restriction along the functor 
$\sigma(-, [0])\colon \Delta^\op\to \Delta^\op$ it has both a left adjoint and a right adjoint --- we give a description 
of the left adjoint in Corollary~\ref{left adjoint of Dec_0} below.  

The comonad $\Dec_0$ gives rise to a simplicial resolution of any object $X\in s\cC$, in other 
words it gives rise to a bisimplicial object in $\cC$ which, when thought of as a simplicial 
object in $s\cC$, has as its object of $n$-simplices the simplicial object $\Dec_nX := (\Dec_0)^{n+1}X$.  
This bisimplicial object is denoted by 
$\Dec\, X$.  Its object of $(m,n)$-bisimplices is $X_{m+n+1}$.  Clearly this construction extends 
to define a functor 
\begin{equation}
\label{total dec}
\Dec\colon s\cC\to ss\cC.  
\end{equation}
This functor can also be described as 
restriction along the opposite 
functor of the ordinal sum map $\sigma\colon \Delta\times \Delta\to \Delta$.   
The bisimplicial object $\Dec\, X$ was introduced by Illusie in \cite{Illusie2} and is called the 
{\em total d\'{e}calage} of $X$.  For more details we refer to \cite{CR,Ste1}.

When $\cC$ has sufficiently many limits and colimits, the functor 
$\sigma^* = \Dec\colon s\cC\to ss\cC$ has both a left adjoint $\sigma_!$ and a right adjoint $\sigma_*$.  
We will later need some elementary facts about the left adjoint of $\Dec$ when $\cC = \Set$.  
In this case the left adjoint of $\Dec$ is closely related to the join operation 
on simplicial sets (see \cite{Ehlers-Porter,Joyal,Lurie}).  
To see this first observe that ordinal sum induces a map $\sigma_a^*\colon s\Set_a\to ss\Set_a$, 
where $s\Set_a = [\Delta_a^\op,\Set]$ denotes the category of augmented simplicial sets and 
$ss\Set_a = [\Delta_a^\op\times \Delta_a^\op,\Set]$ 
denotes the category of biaugmented bisimplicial sets, and that we have a commutative diagram 
\[
\xymatrix{ 
s\Set_a \ar[d]_-{i^*} \ar[r]^-{\sigma_a^*} & 
ss\Set_a \ar[d]^-{(i\times i)^*}                            \\ 
s\Set \ar[r]_-{\sigma^*} & 
ss\Set.  } 
\]
Here $i^*$ and $(i\times i)^*$ denote the functors induced by the inclusions 
$i\colon \Delta\subset \Delta_a$ and $i\times i\colon \Delta\times \Delta\subset \Delta_a
\times \Delta_a$ respectively.  
Let $i_!$ and $(i\times i)_!$ denote the functors left adjoint to $i^*$ and $(i\times i)^*$ 
respectively.  Thus if $S$ is a simplicial set then $i_!(S)$ is the augmented simplicial 
set with $i_!S(-1) = \pi_0(S)$  and similarly for 
$(i\times i)_!$.   
It follows that we have a natural isomorphism of functors 
\[
i_!\sigma_! = (\sigma_a)_! (i\times i)_!
\]
and hence 
\[
\sigma_! = i^*(\sigma_a)_!(i\times i)_!,  
\]
where $(\sigma_a)_!$ denotes the left adjoint of the functor $\sigma_a^*\colon s\Set_a\to ss\Set_a$ 
induced by restriction along $\sigma_a\colon \Delta_a\times \Delta_a\to \Delta_a$.  
Hence we have the following lemma.  
\begin{lemma}
\label{left adjoint of Dec and joins}
For any connected simplicial sets $K$ and $L$ we have the identity 
\[
\sigma_!(K\Box L) = K\star L, 
\]
natural in $K$ and $L$, where $K\star L$ denotes the join of $K$ and $L$.  
\end{lemma} 

Here $K\Box L$ denotes the {\em box product} of the simplicial sets 
$K$ and $L$: recall that the box product is the functor $\Box\colon s\Set\times s\Set\to 
ss\Set$ which sends $(K,L)$ to the bisimplicial set whose set of $(m,n)$-bisimplices 
is $K_m\times L_n$.  Thus $\Delta^m\Box \Delta^n$ is the classifying bisimplex 
$\Delta^{m,n}$ for example.  Note that $\Box$ extends canonically to a functor 
$\Box\colon s\Set_a\times s\Set_a\to ss\Set_a$.    

\begin{proof} 
Observe that $(i\times i)_!(K\Box L) = i_!(K)\Box i_!(L)$.  Since $K$ and $L$ are connected 
we have $(i\times i)_!(K\Box L) = i_*(K)\Box i_*(L)$, where $i_*$ denotes the functor 
right adjoint to $i^*$, so that $i_*(S)([-1])= 1$ for any simplicial set $S$.  Therefore 
\[
\sigma_!(K\Box L) = i^*(\sigma_a)_!(i_*(K)\Box i_*(L)), 
\]
which is by definition equal to the join $K\star L$ (see \cite{Joyal}).  
\end{proof}

As a corollary we have the following result.  
\begin{corollary} 
\label{left adjoint of Dec_0}
The functor $C = \sigma_!((-)\Box \Delta^0)\colon s\Set\to s\Set$ is left adjoint 
to the functor $\Dec_0$.  In particular when $X$ is a connected simplicial set we have 
\[
CX = \sigma_!(X\Box \Delta^0) = X\star \Delta^0.  
\]
\end{corollary} 
Thus when $X$ is connected, $CX$ is the {\em cone} construction on $X$ 
(see Section III.5 of \cite{GJ}).  
\begin{proof} 
Let $X$ be a simplicial set.  Then for any simplicial set $Y$, there is a sequence of natural 
isomorphisms 
\[ 
s\Set(\sigma_!(X\Box \Delta^0),Y) = ss\Set(X\Box \Delta^0,\Dec Y) = s\Set(X,(\Dec Y)/\Delta^0), 
\]
where we recall \cite{JT} that the functor $(-)\Box \Delta^0\colon s\Set\to ss\Set$ is right adjoint 
to the functor $(-)/\Delta^0\colon ss\Set\to s\Set$ which sends a bisimplicial set $S$ to its 
first row.  Hence there is an isomorphism 
\[
s\Set(\sigma_!(X\Box \Delta^0),Y) = s\Set(X,\Dec_0 Y), 
\]
natural in $X$ and $Y$, which proves the result.  
\end{proof} 
  
\subsection{Generalities on matching objects} 
  
Recall (see VII Proposition 1.21 of \cite{GJ}) that if $X$ is a simplicial object in a 
complete category $\cC$ and $K$ is a simplicial set, then 
we can define an object $M_KX$, a {\em generalized matching object} of $X$, by the formula 
\[
M_KX = \varprojlim_{\Delta^n\to K}X_n.  
\]
In other words $M_KX$ is the limit of the diagram $(\Delta/K)^\op\to \Delta^\op\stackrel{X}{\to} \cC$ 
in $\cC$ on the opposite of the simplex category $\Delta/K$ of $K$.  
This process is functorial in $X$ and so defines a functor $M_K\colon s\cC\to \cC$.  It turns out 
that this functor is right adjoint to the functor $\cC\to s\cC$ which sends an object $A$ of $\cC$ to the simplicial 
object $A\otimes K$.  Hence the functor $M_K$ preserves limits, and therefore sends groups in $s\cC$ to 
groups in $\cC$.  

If we fix a simplicial object $X$, then we obtain a functor $\sSet^\op\to \cC$ which 
sends a simplicial set $K$ to the generalized matching object $M_KX$.  It turns out that this functor is also a 
right adjoint, and hence preserves limits.

We would like to compare the generalized matching objects $M^B_KX$ in $\cC_{/B}$ and 
$M_KX$ in $\cC$ for $X$ an object of 
$\cC_{/B}$.  To do this we first need to recall how limits in $\cC_{/B}$ are constructed.  Recall from \cite{May-Sig} that 
products in $\cC_{/B}$ are constructed as fiber products and that pullbacks 
are constructed by first forming the ordinary pullback in $\cC$, and then equipping it with the 
canonical map to $B$.  More generally, if $X\colon I\to \cC_{/B}$ is a diagram in $\cC_{/B}$ and 
$U(X)\colon I\to \cC$ denotes the underlying diagram in $\cC$, then $\varprojlim_{i\in I}X_i$ is 
defined by the following pullback diagram in $\cC$: 
\[
\xymatrix{ 
\varprojlim_{i\in I}X_i \ar[d] \ar[r] & 
\varprojlim_{i\in I}U(X_i) \ar[d] \\ 
B \ar[r] & \varprojlim_{i\in I} B, }
\]
where $B$ is regarded as a constant diagram on $I$.  
It follows that if $X$ is a simplicial object of $\cC_{/B}$ then 
$M^B_KX$ fits into a pullback diagram 
\begin{equation} 
\label{gen matching object pullback}
\begin{xy} 
(-10,7.5)*+{M^B_KX}="1"; 
(-10,-7.5)*+{B}="2"; 
(10,7.5)*+{M_KX}="3"; 
(10,-7.5)*+{M_KB}="4"; 
{\ar "1";"3"}; 
{\ar "1";"2"}; 
{\ar "2";"4"};
{\ar "3";"4"};
\end{xy}
\end{equation}
where we think of $B$ as a constant simplicial object in $\cC$. 
We have the following well known lemma (see for instance \cite{DHI} Lemma 4.2).
\begin{lemma}
\label{gen matching objects in C/B}
Let $X$ be an object of $\cC_{/B}$.  Then we have 
\[
M^B_{\Lambda^n_k} X = M_{\Lambda^n_k}X 
\]
for all $0\leq k\leq n$ and all $n\geq 1$.  Similarly we have 
\[
M^B_{\partial\Delta^n}X = M_{\partial\Delta^n}X    
\]
for all $n\geq 2$.  
\end{lemma}
\begin{proof}
Since the diagram~\eqref{gen matching object pullback} is a pullback, it is enough to prove that the 
maps $B\to M_{\Lambda^n_k}B$ and $B\to M_{\partial\Delta^n}B$ are isomorphisms, where $B$ is regarded 
as a constant simplicial object in $\cC$.  Therefore, it is enough to prove that the simplex categories 
$\Delta\downarrow\Lambda^n_k$ and $\Delta\downarrow\partial\Delta^n$ are connected.  This follows 
from the fact that the spaces $|\Lambda^n_k|$ and $|\partial\Delta^n|$ are connected, since 
$N(\Delta\downarrow \Lambda^n_k)$ and $N(\Delta\downarrow \partial\Delta^n)$ 
are weakly equivalent to $\Lambda^n_k$ and $\partial\Delta^n$ respectively.      
\end{proof}

The notion of generalized matching object has an analog for bisimplicial objects in $\cC$.  If $K$ is a 
bisimplicial set and $X$ is a bisimplicial object of $\cC$, then the generalized matching object 
$M_KX$ is defined by an analogous limit formula in $\cC$ to that above:
\[
M_KX = \varprojlim_{\Delta^{m,n}\to K} X_{m,n}.
\]
For a fixed bisimplicial set $K$ this defines a functor $M_K\colon ss\cC\to \cC$.  Note that if $K$ and $L$ 
are simplicial sets and $X$ is a bisimplicial object of $\cC$ then we have an isomorphism 
\begin{equation} 
\label{external product and matching objects}
M_{K\Box L} X = M_KM_LX 
\end{equation}
where $M_LX$ denotes the matching object of the simplicial object $X$ of $s\cC$.   

We have the following 
result.  
\begin{lemma} 
\label{adjoint for bisimplicial matching}
Let $K$ be a bisimplicial set.  The functor $M_K\colon ss\cC\to \cC$ has a left adjoint $(-)\otimes K\colon \cC\to ss\cC$.  
\end{lemma}
\begin{proof}
$(-)\otimes K\colon \cC\to ss\cC$ is the functor whose value on an object $X$ of $\cC$ is the bisimplicial object 
$(X\otimes K)$ whose object of $(m,n)$-bisimplices is 
\[
(X\otimes K)_{m,n} = \coprod_{k\in K_{m,n}} X.
\]
A map $\a\colon ([m],[n])\to ([p],[q])$ induces a map $\a^*\colon (X\otimes K)_{p,q}\to (X\otimes K)_{m,n}$ in the 
obvious fashion.  The proof that this functor is left adjoint to $M_K$ is   
entirely analogous to the proof of VII Proposition 1.21 of \cite{GJ}, to which we 
refer the reader for details.  
\end{proof}
The following proposition describes a useful property of 
matching objects for bisimplicial objects that we shall put to use 
in Section~\ref{Proof of first Proposition}.  
\begin{proposition}
\label{bisimplicial matching and Dec}
Let $X$ be a simplicial object in $\cC$ and let $K$ be a bisimplicial set.  Then there is an isomorphism 
\[
M_{\sigma_!K}X \cong M_K \Dec\, X, 
\]
natural in $X$ and $K$, where $\sigma_!\colon ss\Set\to s\Set$ denotes the left adjoint to $\Dec\colon s\Set\to ss\Set$.  
In particular when $K$ and $L$ are connected simplicial sets, then we have an isomorphism 
\[
M_{K\Box L}\Dec\, X\cong M_{K\star L}X, 
\]
natural in $K$, $L$ and $X$, where $K\star L$ denotes the join of $K$ and $L$.    
\end{proposition}  
As an immediate corollary to this Proposition we have the following result. 
\begin{corollary} 
\label{cor:cones and matching objects}
Let $X$ be a simplicial object in $\cC$ and let $K$ be a simplicial set.  Then there is an isomorphism 
\[
M_K\Dec_0X\cong M_{CK}X, 
\]
natural in $X$ and $K$, where $CK$ denotes the cone construction on $K$.   
\end{corollary}
\begin{proof} 
In Proposition~\ref{bisimplicial matching and Dec} take the bisimplicial set $K$ to be $K\Box \Delta^0$.  Then by 
Corollary~\ref{left adjoint of Dec_0} we have natural 
isomorphisms 
\begin{align*} 
M_{\sigma_!(K\Box \Delta^0)}X & \cong M_{K\Box \Delta^0}\Dec\, X \\ 
& \cong M_KM_{\Delta^0}\Dec\, X \\ 
& \cong M_K \Dec_0 X 
\end{align*} 
where in the second isomorphism we have used~\eqref{external product and matching objects}.  To complete the 
proof we recall that $CK = \sigma_!(K\Box\Delta^0)$.  
\end{proof}
\begin{proof}[Proof of Proposition~\ref{bisimplicial matching and Dec}]
Let $X$ and $Y$ be objects of $\cC$ and let $K$ be a bisimplicial set.  By hypothesis $\cC$ is 
cocomplete and so the functor $\Dec = \sigma^*\colon s\cC\to ss\cC$ has a left adjoint 
$\sigma_!\colon ss\cC\to s\cC$ given as a left Kan extension, so that 
\[
(\sigma_!X)_n =   \varinjlim_{\sigma([p],[q])\to [n]} X_{p,q}.  
\]
Therefore 
we can use Lemma~\ref{adjoint for bisimplicial matching} to establish the following series of natural isomorphisms.  
\begin{align*} 
\cC(X,M_K\Dec\, Y) & \cong ss\cC(X\otimes K,\Dec\, Y) \\
& \cong s\cC(\sigma_!(X\otimes K),Y) \\ 
& \cong s\cC(X\otimes \sigma_!K,Y) \\ 
& \cong \cC(X,M_{\sigma_!K}Y) 
\end{align*} 
where we have used the fact (to be proved shortly below) that there is a natural 
isomorphism $\sigma_!(-\otimes K)\cong (-)\otimes \sigma_!(K)$.  It follows that 
$M_K\Dec\, Y$ and $M_{\sigma_!K}Y$ are isomorphic.  

To see that there is such a natural isomorphism 
as above we use the fact that we can interchange colimits to perform the following computation: 
\begin{align*} 
\sigma_!(X\otimes K)_n & = \varinjlim_{\sigma([p], [q])\to [n]} (X\otimes K)_{p,q} \\ 
& = \varinjlim_{\sigma([p], [q])\to [n]} \coprod_{K_{p,q}} X \\ 
& = \coprod_{\sigma_!K_n} X \\ 
& = (X\otimes \sigma_!K)_n.  
\end{align*} 
These isomorphisms are natural in $n$ and hence the simplicial 
objects $\sigma_!(X\otimes K)$ and $X\otimes \sigma_!K$ are 
isomorphic.    
\end{proof}

\section{Internal local simplicial homotopy theory} 
\label{sec:internal}

In this section we study some aspects of local homotopy 
theory internal to a category $\cC$ equipped with a Grothendieck 
pretopology.  Our treatment follows closely the discussion in 
\cite{Henriques} and the classic \cite{Jardine1}.  

\subsection{Grothendieck topologies} 
\label{sec:groth tops}
Let $\cC$ be a category with all limits and colimits and which comes equipped with a {\em Grothendieck pretopology}.  
Thus for every object $U$ of $\cC$ there is assigned a class of 
families of morphisms $(U_i\to U)_{i\in I}$, called {\em 
covering families} which satisfy three axioms, namely: isomorphisms are covering 
families, covering families are stable under pullback and transitive under composition.  
For example, to say that covering families are stable under pullback 
means that 
if $(U_i\to U)_{i\in I}$ is a covering family, and $V\to U$ is a 
morphism in $\cC$, then $(V\times_U U_i\to V)_{i\in I}$ is also a covering family.
For more details we refer to Definition 2 in Section III.2 of \cite{MM}. 

We'll often refer to a covering family $(U_i\to U)_{i\in I}$ as a {\em cover} of $U$.  It is clear that if $\cC$ is equipped 
with a Grothendieck pretopology, then for any object $B$ of $\cC$, the slice category $\cC_{/B}$ inherits a natural Grothendieck 
pretopology: a family of morphisms $(U_i\to U)_{i\in I}$ in $\cC_{/B}$ is a covering family in $\cC_{/B}$ if and only if 
the underlying family of morphisms $(U_i\to U)_{i\in I}$ is a covering family in $\cC$.  

We will say that a map $f\colon X\to Y$ in $\cC$ is a {\em local epimorphism} if there exists a covering family 
$(Y_i\to Y)_{i\in I}$ with the property that for each $i\in I$ there is a section of the 
induced map $Y_i\times_Y X\to Y_i$.  Note that local epimorphisms are closed 
under retracts.    

We will also need to talk about presheaves represented by 
objects of $\cC$.  When $\cC$ is small we will sometimes write 
$\widehat{\cC}$ for the category $[\cC^\op,\Set]$ of set-valued presheaves on $\cC$.  
In general however, we will run into set theoretic 
difficulties in trying to consider the category of 
presheaves on $\cC$.  Rather than confronting these 
difficulties we will suppose that in addition to $\cC$, there 
exists a small subcategory $\cC_0\subset \cC$ such that 
$\cC_0$ contains a terminal object and is equipped 
with a Grothendieck topology in such a way that if 
$(U_i\to U)_{i\in I}$ is a covering family of $U$ in $\cC_0$, 
then $(U_i\to U)_{i\in I}$ is also a covering family in 
$\cC$.  

Then every object $X$ of $\cC$ represents a 
presheaf $\hat{X}$ on $\cC_0$.  If $f\colon X\to Y$ is a map 
in $\cC$ then we write $\hat{f}\colon \hat{X}\to \hat{Y}$ 
for the induced map of representable presheaves on 
$\cC_0$.  Note that if $f\colon X\to Y$ is a local epimorphism 
in $\cC$, then $\hat{f}\colon \hat{X}\to \hat{Y}$ is a 
local epimorphism of presheaves on 
$\cC_0$ in the sense that for every section $y\in \hat{Y}(U)$, there is a covering 
family $(U_i\to U)_{i\in I}$ and elements $x_i\in \hat{X}(U)$ 
such that $\hat{f}(x_i) = y|_{U_i}$ for each $i\in I$.      

As our main example we will consider the case where $\cC = \cK$, 
the category of $k$-spaces, where a covering family 
of a space $U$ in $\cK$ is understood to mean as usual a collection 
of maps of the form $(U_i\to U)_{i\in I}$, where $(U_i)_{i\in I}$ forms an 
open cover of $X$.  For this definition to make sense, we need to know 
that every open subspace of a $k$-space is itself a $k$-space.  However this follows 
from results of \cite{Vogt}: (with the notations of that paper) Theorem 5.1 of \cite{Vogt} shows that 
$\cK$ (called $\mathscr{KG}$ in \cite{Vogt}) satisfies Axiom 1*, and therefore  
Proposition 2.4 of \cite{Vogt} shows that the relative topology on any open 
subset of a $k$-space coincides with the topology of its $k$-ification.  

With this choice of $\cC$, we will be interested in the following small subcategory 
$\cC_0$ of $\cC$.  Let $B$ be a paracompact space in $\cK$ and let 
$\cC_0 = \Open(B)$ be the subcategory of $\cK$ generated by all open subsets of 
$B$.  It is clear that in this case $\cC_0$ has a Grothendieck 
topology compatible with the open cover topology on $\cC$.

\subsection{Local homotopy theory of simplicial presheaves} 
\label{simp pshvs} 

Suppose that $\cC$ is a small category.  
Recall that there are two model structures on the category of diagrams 
$s\widehat{\cC} = [\cC^\op,\sSet]$.  The weak equivalences for each of these 
model structures are the {\em object-wise weak homotopy equivalences}, 
i.e.\ the maps $X\to Y$ in $s\widehat{\cC}$ such that 
$X(C)\to Y(C)$ is a weak homotopy equivalence of simplicial sets for all $C\in \cC$.  
The {\em injective model structure} has as its cofibrations 
the monomorphisms, with the fibrations determined by the RLP 
with respect to the trivial cofibrations.  The {\em projective model 
structure} has as its fibrations the 
object-wise Kan fibrations with the cofibrations determined by the LLP 
with respect to the trivial fibrations.  To distinguish between the 
injective and projective model structures we will write $s\Pre(\cC)$ 
for the injective model structure on the category of simplicial presheaves 
and, following \cite{Dugger-Univ}, we will write $U\cC$ for the 
projective model structure (in \cite{Dugger-Univ} this is called the {\em universal} model 
structure).      

For us, the projective model structure has a slight advantage 
over the injective model structure: the fibrant objects are easy to 
understand and there is a criterion to detect cofibrant 
objects which is relatively easy to check in practice.  
Namely, we have the following result from \cite{Dugger-Univ}.  
\begin{proposition}[Dugger \cite{Dugger-Univ}] 
\label{Dugger-split}
A simplicial presheaf $X$ is cofibrant in $U\cC$ if it is split, and is a degree-wise 
coproduct of representables.  
\end{proposition} 

Recall that a simplicial object $X$ is said to be {\em split} if 
for all $n\geq 0$ there exist subobjects $U_n\subset X_n$ such that the canonical map 
\[
\coprod_{\sigma\colon [n]\twoheadrightarrow [m] } U_{\sigma} \to X_n 
\]
is an isomorphism, where $U_{\sigma}$ denotes a copy of $U_m$ and the map $U_{\sigma}\to X_n$ is the  
composition $U_m \to X_m \xrightarrow{\sigma^*} X_n$.  In fact, Dugger 
shows in \cite{Dugger-Univ} that there is a very convenient cofibrant replacement 
functor in $U\cC$ which replaces any simplicial presheaf with a split one.    

Now suppose that $\cC$ is a small category equipped with a Grothendieck topology.  
In the paper \cite{Jardine1} the category $s\widehat{\cC}$ of simplicial presheaves on 
$\cC$ was equipped with the structure of a model category.  
Following \cite{DHI} we will write $s\Pre(\cC)_{\mathcal{L}}$ 
for this model category.  The weak equivalences of this model structure 
are defined in terms of certain sheaves of simplicial homotopy groups as 
we now recall.  

Let $X$ be a simplicial presheaf on $\cC$.  If $n\geq 1$ then 
for any object $C$ of $\cC$ and any vertex $v\in X_0(C)$, Jardine defines 
sheaves $\pi_n(X_C,v)$ on $\cC_{/C}$ (here $X_C$ denotes the 
restriction of $X$ to $\cC_{/C}$) as the sheaves associated to the 
presheaves 
\[
(C'\to C)\mapsto \pi_n(X(C'),v). 
\]
Similarly the sheaf $\pi_0(X_C)$ on $\cC_{/C}$ is defined to be the sheaf associated 
to the presheaf 
\[
(C'\to C)\mapsto \pi_0(X(C')).  
\]
A map $f\colon X\to Y$ in $s\widehat{\cC}$ is called a {\em local weak equivalence} 
if the induced maps $\pi_0(X_C)\to \pi_0(Y_C)$ and $\pi_n(X_C,v)\to \pi_n(Y_C,f(v))$ 
are isomorphisms of sheaves on $\cC_{/C}$ for all objects $C$ of $\cC$ and all choices of vertices $v\in X_0(C)$.  
In \cite{Jardine1} it is proven that the local weak equivalences are the weak equivalences 
for a model structure on $s\widehat{\cC}$ whose cofibrations are the monomorphisms.   

We will only be interested in a very special case of this general theory --- namely 
when $\cC$ is the category $\Open(B)$ of open subsets of a paracompact space $B$ 
in $\cK$, as described above.  We will write $s\Pre(B)$ 
for the category $s\widehat{\cC}$ of simplicial presheaves on $\cC$ and say 
that an object $X\in s\Pre(B)$ is a simplicial presheaf on $B$ in this case.  
In this case the local weak equivalences are much easier to describe: a map $f\colon X\to Y$ 
in $s\Pre(B)$ is a local weak equivalence if and only if the induced map on stalks 
$f_p\colon X_p\to Y_p$ is a weak homotopy equivalence for all points $p$ in $B$.  

An important fact proven in \cite{DHI} is that the Jardine model structure $s\Pre(\cC)_{\mathcal{L}}$ is 
a localization of the injective model structure $s\Pre(\cC)$ at a certain collection of 
maps.  We recall some of the details here as it will be important for us later.   
Let $C$ be an object of $\cC$ and write $C$ also for the corresponding representable 
presheaf $\hat{C}$ on $\cC$.  Recall that a map $U\to C$ in $s\widehat{\cC}$ is called 
a {\em hypercover} of $C$ if each $U_n$ is a coproduct of representables and the maps 
\[
U_n\to M^C_{\partial\Delta^n}U 
\]
are local epimorphisms of presheaves on $\cC$.  Here the matching objects are computed 
in the category $s\widehat{\cC}/C$ as described above.  A hypercover $U\to C$ is said to be {\em split} 
if the underlying simplicial presheaf $U$ is split.  One of the main theorems of \cite{DHI} 
is the following.  
\begin{theorem}[\cite{DHI}] 
Let $S$ be the class of all hypercovers on $\cC$.  Then the localization $s\Pre(\cC)/S$ 
exists and coincides with $s\Pre(\cC)_\mathcal{L}$.  Moreover the localization $U\cC_{\mathcal{L}} = U\cC/S$ 
exists and there is a Quillen equivalence 
\[
s\Pre(\cC)_{\mathcal{L}} 
\leftrightarrows 
U\cC_{\mathcal{L}}. 
\]  
\end{theorem}   
As pointed out in \cite{DHI}, the localized projective model structure $U\cC_{\mathcal{L}}$ 
has an advantage over $s\Pre(\cC)_{\mathcal{L}}$ in that the fibrant objects are easier to 
describe and there is an explicit formula for cofibrant replacement.  
In \cite{DHI}, a simplicial presheaf $X$ is said to {\em satisfy descent} for a 
hypercover $U\to C$ if the natural map 
\[
F(C) \to 
\hoinvlim_{n} F(U_n) 
\]
is a weak homotopy equivalence of simplicial sets.  In \cite{DHI} the following characterization 
of the fibrant objects in $U\cC_{\mathcal{L}}$ is obtained.  
\begin{theorem}[\cite{DHI}] 
A simplicial presheaf $X$ is fibrant in $U\cC_{\mathcal{L}}$ if and only if 
\begin{enumerate} 
\item $X$ is objectwise fibrant, i.e.\ $X(C)$ is a Kan complex for all objects $C$ of $\cC$, 

\item $X$ satisfies descent for all hypercovers.  
\end{enumerate}
\end{theorem}  
In general it is difficult to compute the set of homotopy classes $[X,Y]$ in $U\cC_{\mathcal{L}}$, 
since it is difficult to compute fibrant replacements in $U\cC_{\mathcal{L}}$.  The  
generalized Verdier hypercovering theorem \cite{Brown, DHI, Jardine1} allows one to compute some invariants of $[X,Y]$, 
at least when $Y$ is locally fibrant in the following well known sense.  

Suppose that $i\colon K\to L$ is a map of simplicial sets.  A map $f\colon X\to Y$ in 
$s\widehat{\cC}$ is said to have {\em local liftings} relative to $i\colon K\to L$ if for all 
objects $C$ in $\cC$ and all diagrams in $s\widehat{\cC}$ of the form 
\[
\xymatrix{ 
K\otimes C \ar[d]_-i \ar[r] & X \ar[d]^-f \\ 
L\otimes C \ar[r] & Y,} 
\]
there exists a covering sieve $R\subset \cC(-,C)$ such that for every map 
$\phi\colon U\to C$ in $R$, the diagram obtained by restricting along $\phi$ has a lift 
$L\otimes U\to X$.  A map $f\colon X\to Y$ in $s\widehat{\cC}$ is said to be a 
{\em local fibration} if $f$ has local liftings relative to $\Lambda^n_k\subset \Delta^n$ for 
all $0\leq k\leq n$ and all $n\geq 1$.  An object $X$ in $s\widehat{\cC}$ is 
said to be {\em locally fibrant} if the canonical map $X\to 1$ is a local fibration.  
It can be shown (see \cite{DI,Jardine1}) 
that a map $f\colon X\to Y$ is a locally acyclic local fibration (i.e.\ a local weak equivalence and 
a local fibration) if and only if $f$ has local liftings relative to $\partial\Delta^n\subset \Delta^n$ 
for all $n\geq 0$.  Thus if $C$ is an object of $\cC$ then a map $f\colon U\to C$ in $s\widehat{\cC}$ 
is a hypercover if and only if $f$ is a locally acyclic local fibration and $U_n$ is a coproduct 
of representables for all $n\geq 0$.  

When $\cC = \Open(B)$, a map $f\colon X\to Y$ in $s\widehat{\cC}$ is a local fibration 
if and only if the induced map on stalks $f_p\colon X_p\to Y_p$ is a Kan fibration for all points 
$p$ of $M$.  If $f\colon X\to Y$ is an objectwise Kan fibration then $f$ is a local fibration, but not 
conversely.    

To state the hypercovering theorem we need some notation.  
For an object $C$ of $\cC$, let $HR(C)$ denote the full subcategory of $s\widehat{\cC}$ 
whose objects are the hypercovers of $C$.  We write $\pi HR(C)$ for the category with the 
same objects as $HR(C)$ but whose morphisms are simplicial homotopy classes of 
morphisms in $HR(C)$.  We have the following result.  
\begin{theorem}[generalized Verdier hypercovering theorem \cite{Brown,DHI,Jardine1}] 
\label{verdier hypercovering theorem}
Let $X$ be a locally fibrant simplicial presheaf on $\cC$ and let $C$ be an object of $\cC$.  Then there 
is an isomorphism 
\[
[C,X] = \varinjlim_{U\in \pi HR(C)} \pi(U,X) 
\]
in $\Ho(U\cC_{\mathcal{L}})$ where $\pi(-,-)$ denotes simplicial homotopy classes of maps.  
\end{theorem}  
As pointed out in \cite{DHI}, the colimit above could just as well be taken over any full subcategory 
of $\pi HR(C)$ whose objects belong to a {\em dense} set of hypercovers.  In the case of main 
interest for us, namely when $\cC = \Open(B)$, this means that the colimit above could be 
taken over the full subcategory of $\pi HR(C)$ whose objects are the split hypercovers.  
Note that since $U\cC_{\mathcal{L}}$ is a left Bousfield localization, every cofibrant object of $U\cC$ is automatically cofibrant in 
$U\cC_{\mathcal{L}}$.  In particular every split hypercover is cofibrant in $U\cC_{\mathcal{L}}$ (Proposition~\ref{Dugger-split}).  
As a consequence we have the following simple observation.  
\begin{lemma} 
\label{surjectivity}
Let $p\colon X\to Y$ be an objectwise Kan fibration of 
simplicial presheaves on $B$ and suppose that $Y$ is locally fibrant.  
Then the map 
\[
[1,X]\to [1,Y] 
\]
in $\Ho(U\cC_{\mathcal{L}})$ is surjective, where $1$ denotes the terminal object.  
\end{lemma}
\begin{proof} 
Since $Y$ is locally fibrant and $p$ is a local fibration, $X$ is locally fibrant.  Since $p\colon X\to Y$ is a 
fibration in $U\cC$, for any split hypercover $U$ of $1$, the map 
\[
\map(U,X)\to \map(U,Y) 
\]
is a Kan fibration and hence is surjective on path components.  Since filtered colimits 
preserve surjections it follows that the map 
\[
\varinjlim_{U\in \pi HR_s(1)} \pi_0\map(U,X)\to 
\varinjlim_{U\in \pi HR_s(1)} \pi_0\map(U,Y)
\]
is surjective, where $\pi HR_s(1)$ denotes the full subcategory of $\pi HR(1)$ consisting 
of the split hypercovers.  By the remarks above this map is isomorphic to the map 
$[1,X]\to [1,Y]$ in $\Ho(U\cC_{\mathcal{L}})$.    
\end{proof}

\subsection{Internal local fibrations} 

Recall the following definition from \cite{Henriques}.  
\begin{definition}[\cite{Henriques}] 
\label{locally fibrant}
We say that a morphism $f\colon X\to Y$ in 
$s\cC$ is a {\em local fibration} if for every $n\geq 1$ and every $0\leq k\leq n$ the map 
\begin{equation} 
\label{horn filling map for fibrns} 
X_n\to Y_n\times_{M_{\Lambda^n_k}Y} M_{\Lambda^n_k}X 
\end{equation}
is a local epimorphism.  We say that a simplicial object $X$ in $\cC$ 
is {\em locally fibrant} if the canonical map $X\to 1$ to the terminal simplicial object is a local fibration: thus 
$X$ is locally fibrant if for every $n\geq 1$ and every $0\leq k\leq n$ the map 
\begin{equation}
\label{horn filling map} 
X_n\to  M_{\Lambda^n_k}X 
\end{equation}
is a local epimorphism.  
We say that a morphism $f\colon X\to Y$ in $s\cC$ is a {\em locally acyclic local fibration} if for every $n\geq 0$ the map 
\[
X_n\to Y_n\times_{M_{\partial \Delta^n}Y}M_{\partial\Delta^n}X 
\]
is a local epimorphism.  Similarly we say that a simplicial object $X$ of $\cC$ is {\em locally acyclic} 
if the map $X\to 1$ to the terminal simplicial object is an acyclic local fibration: thus $X$ is locally acyclic 
if for every $n\geq 0$ the map 
\[
X_n\to M_{\partial\Delta^n}X
\]
is a local epimorphism.    
\end{definition} 

Clearly if $f\colon X\to Y$ is a local fibration in $s\cC$ then the induced map 
$\hat{f}\colon \hat{X}\to \hat{Y}$ on representable simplicial presheaves in $s\Pre(\cC_0)$ 
is a local fibration.  An analogous remark applies if $f\colon X\to Y$ is an acyclic local fibration in $s\cC$.

\begin{lemma}
Let $B$ be an object of $\cC$ and let $f\colon X\to Y$ be a morphism in $s\cC_{/B}$.  Then $f$ is 
a local fibration in $s\cC_{/B}$ if and only if 
the underlying map $f\colon X\to Y$ is a 
local fibration in $s\cC$.  In particular an object $X$ in 
$s\cC_{/B}$ is locally fibrant in $s\cC_{/B}$ if and only if 
the underlying object $X$ is locally fibrant in 
$s\cC$. 
\end{lemma}
\begin{proof} 
$f\colon X\to Y$ is a local fibration in $s\cC_{/B}$ if and only if 
\[
X_n\to M^B_{\Lambda^n_k}X\times_{M^B_{\Lambda^n_k}Y} Y_n 
\]
is a local epimorphism in $\cC_{/B}$ for all $0\leq k\leq n$ and all $n\geq 1$.  Hence 
Lemma~\ref{gen matching objects in C/B} implies that $f\colon X\to Y$ is a local fibration 
in $s\cC$ if and only if 
\[
X_n\to M_{\Lambda^n_k}X\times_{M_{\Lambda^n_k}Y} Y_n 
\]
is a local epimorphism in $\cC$ for all $0\leq k\leq n$ and all $n\geq 1$, i.e.\ $f\colon X\to Y$ is a local fibration in $s\cC$.   
\end{proof}

\begin{lemma}
\label{lem:grp} 
Let $B$ be an object of $\cC$ and let $G$ be a group in $s\cC_{/B}$.  Then $G$ is locally fibrant.
\end{lemma} 
\begin{proof}
Lemma~\ref{gen matching objects in C/B} shows that it is enough to check that the underlying map 
$G\to B$ in $s\cC$ is locally fibrant.  First let us observe that if $X$ is a group object in $s\Set/Y$, where 
$Y$ is a constant simplicial set, then the map $X\to Y$ is a Kan fibration.    
Returning to the case at hand, for any object $A$ in $\cC$, if we set $X = \cC(A,G)$ and $Y = \cC(A,B)$, then 
$X$ is a group object in $s\Set/Y$ whose group object of $n$-simplices is 
$\cC(A,G_n)$.  Therefore by the observation just made, 
\[
\cC(A,G_n)\to M_{\Lambda^n_k}\cC(A,G) 
\]
is surjective for all $0\leq k\leq n$ and all $n\geq 1$.  Since $\cC$ is complete we have 
\[
M_{\Lambda^n_k}\cC(A,G) = \cC(A,M_{\Lambda^n_k} G).  
\]
Therefore 
\[
\cC(A,G_n)\to \cC(A,M_{\Lambda^n_k} G) 
\]
is surjective for all $0\leq k\leq n$ and all $n\geq 1$.  Taking $A = M_{\Lambda^n_k} G$ gives the result.  
\end{proof}

\begin{corollary}
\label{corr:homo to constant grp}
Suppose that $f\colon G\to H$ is a homomorphism of group objects in $s\cC_{/B}$ where $H$ is the 
constant simplicial object associated to a group object $H$ in $\cC_{/B}$.  Then $f\colon G\to H$ is a local fibration.   
\end{corollary} 
\begin{proof}
We need to check that the maps 
\[
G_n\to M_{\Lambda^n_k}G\times_{M_{\Lambda^n_k}H} H_n 
\]
are local epimorphisms in $\cC$ for all $0\leq k\leq n$ and all $n\geq 1$.  Since $H$ is constant we have 
\[
M_{\Lambda^n_k}H = H_n 
\]
and so the result follows by Lemma~\ref{lem:grp}.   
\end{proof}

\begin{corollary} 
\label{corr:surj homo of grps}
Suppose that $f\colon G\to H$ is a homomorphism between group objects in $s\cC_{/B}$ such that $f_n\colon G_n\to H_n$ 
is a local epimorphism for all $n\geq 0$.  Then $f$ is a local fibration.  
\end{corollary}
\begin{proof}
We leave the proof of this to the reader.  
\end{proof}

\begin{lemma} 
\label{henriques lemma}
Let $B$ be an object of $\cC$ and let $f\colon X\to Y$ be a local fibration in $\cC_{/B}$.  Let $K\subset L$ be 
a weak equivalence of finite simplicial sets.  Then the map 
\[ 
M_L X\to M_K X\times_{M_K Y}M_L Y
\] 
is a local epimorphism in $\cC_{/B}$.  
\end{lemma} 
\begin{proof} 
The proof of this is just an obvious reformulation of the proof of the corresponding fact for simplicial sets.  
In more detail, observe that $f\colon X\to Y$ has the local 
RLP with respect to any map of the form 
$J\to J\cup_{\Lambda^n_k} \Delta^n$ 
for any simplicial set $J$.  Similarly 
observe that if $f\colon X\to Y$ has the local 
RLP with respect to $I\to J$ and $K\to L$ 
is a retract of $I\to J$, then $f\colon X\to Y$ 
has the local RLP with respect to $K\to L$.      
Finally, $K\to L$ is a retract of a finite 
composition of pushouts of maps of the 
form $\Lambda^n_k\to\Delta^n$. 
\end{proof} 

\subsection{A criterion for local acyclicity}

Our first main technical result gives a criterion to detect when certain homomorphisms between group objects in $s\cC_{/B}$ are 
locally acyclic local fibrations.     
\begin{proposition} 
\label{contractibility propn}  
Let $B$ be an object of $\cC$.  
Suppose that $H$ is a group object in $\cC_{/B}$, and that 
$f\colon G\to H$ is a homomorphism between group objects in $s\cC_{/B}$, where $H$ is regarded as a constant 
simplicial object in $\cC_{/B}$.  Then $f\colon G\to H$ is a locally acyclic local fibration in $s\cC_{/B}$ 
for the induced topology on $\cC_{/B}$ if the following 
conditions hold:
\begin{enumerate}
\item 
$G_0\to H$ is a local epimorphism in $\cC_{/B}$,  

\item $G_1\to G_0\times_H G_0$ is a local epimorphism in $\cC_{/B}$, 

\item the underlying map $G_n\to M_{\partial\Delta^n}G$ is an effective epimorphism in $\cC$ for all $n\geq 2$.  
\end{enumerate}
\end{proposition} 
\begin{proof}
We need to prove that 
\[
G_n\to M^B_{\partial\Delta^n}G\times_{M^B_{\partial\Delta^n}H}H_n
\]
is a local epimorphism in $\cC_{/B}$ for all $n\geq 0$.  When $n=0$ this reduces to the requirement that $G_0\to H$ is a local 
epimorphism.  When $n=1$ we have $M^B_{\partial\Delta^1}G = G_0\times_B G_0$ and $M^B_{\partial\Delta^1}H = H\times_B H$, so we need to prove that 
\[
G_1 \to (G_0\times_B G_0)\times_{H\times_B H} H
\]
is a local epimorphism in $\cC_{/B}$ --- which reduces to the requirement that $G_1\to G_0\times_H G_0$ is a local epimorphism.  

Therefore the proposition reduces to the claim that under the given 
hypotheses $\partial_n\colon G_n\to M_{\partial\Delta^n}G$ is a local epimorphism in $\cC$ for all $n\geq 2$.  
We will show that in fact the given hypotheses imply that this 
map has a section.  

Let $\ker(\partial_n)$ denote the group object $1\times_{M_{\partial\Delta^n}G}G_n$ in $\cC_{/B}$.  We will 
first construct a $\ker(\partial_n)$ equivariant retraction $r_n\colon G_n\to \ker(\partial_n)$ in $\cC_{/B}$.  To do this 
we will show that it suffices to construct an analogous retraction for simplicial sets.  By the 
Yoneda Lemma, it suffices to find a $\cC(-,\ker(\partial_n))$ equivariant retraction 
\[
r\colon \cC(-,G_n)\to \cC(-,\ker(\partial_n)) 
\]
of the map of representable presheaves 
\[
\cC(-,\ker(\partial_n))\to \cC(-,G_n) 
\]
induced by the homomorphism of group objects $\ker(\partial_n)\to G_n$.  Let $A$ be an object of $\cC$.  Then $\cC(A,G)$ 
is a group object in $s\Set/K$, where $K$ denotes the constant simplicial set $\cC(A,B)$.  The 
set of $n$-simplices of $\cC(A,G)$ is the group $\cC(A,G_n)$ over $K$.  Since the functor 
$\cC(A,-)$ preserves limits, we have an isomorphism 
\[
\cC(A,M_{\partial\Delta^n}G) = M_{\partial\Delta^n}\cC(A,G)  
\]
in $s\Set/K$.
Therefore the homomorphism $\cC(A,G_n)\to \cC(A,M_{\partial\Delta^n}G)$ is isomorphic 
to the homomorphism $\cC(A,G)_n\to M_{\partial\Delta^n}\cC(A,G)$ and hence 
$\cC(A,\ker(\partial_n))$ is isomorphic to the kernel of the homomorphism 
$\cC(A,G)_n\to M_{\partial\Delta^n}\cC(A,G)$ of group objects in $s\Set/K$.  Hence it suffices to prove that for any group 
$G$ in $s\Set/K$, there is a $\ker(\partial_n)$ equivariant retraction $r_n\colon G_n\to \ker(\partial_n)$ which is 
natural with respect to homomorphisms $G\to G'$ of group objects in $\sSet_{/K}$.    

To prove this, first observe that if $G$ is a group 
object in $\sSet_{/K}$ then $\ker(\partial_n)$ is the set of $n$-simplices $g\in G_n$ such that 
$d_i(g) = 1$ for all $0\leq i\leq n$.  Define a map 
\begin{gather*} 
G_n\to \ker(d_0) \\ 
g\mapsto s_0d_0(g)^{-1}g. 
\end{gather*} 
Note that this map is the identity on $\ker(\partial_n)$ and is $\ker(\partial_n)$ equivariant.  
Now let $0\leq m<n$ and define a map 
\[ 
\bigcap_{0\leq i\leq m} \ker(d_i) \to \bigcap_{0\leq i\leq m+1}\ker(d_i)  
\]
by sending $g\mapsto s_{m+1}d_{m+1}(g)^{-1}g$.  
To see that the image of this map is contained in the subgroup $\bigcap_{0\leq i\leq m+1}\ker(d_i)$, we note 
that if $i<m+1$ then 
\[
d_is_{m+1}d_{m+1}(g)  = s_md_id_{m+1}(g)  
 = s_md_md_i(g)  = 1.  
\]
Note that this map again restricts to the identity on $\ker(\partial_n)$ and is 
$\ker(\partial_n)$ equivariant.  We have constructed a sequence of $\ker(\partial_n)$ 
equivariant maps 
\[
G_n \to \ker(d_0) \to \ker(d_0)\cap \ker(d_1) \to \cdots \to \ker(\partial_n)
\]
which restrict to the identity on $\ker(\partial_n)$.  Composing these maps 
gives the required retraction $r_n$ which is clearly natural with respect to 
homomorphisms $G\to G'$ of group objects in $\sSet_{/K}$.  

To complete the proof we need to show that the existence of the $\ker(\partial_n)$ equivariant 
retraction $r_n\colon G_n\to \ker(\partial_n)$ in $\cC_{/B}$ implies the existence of a section of 
$\partial_n\colon G_n\to M_{\partial\Delta^n}G$ in $\cC_{/B}$.  For this we have the following Lemma.   

\begin{lemma} 
Let $\cD$ be a category with finite limits and let 
$\phi\colon G\to H$ be a homomorphism between 
group objects in $\cD$.  Suppose that $\ker(\phi)$ 
is a retract of $G$.  If $\phi\colon G\to H$ is an 
effective epimorphism in $\cD$ then there is a section 
of $\phi$.  
\end{lemma}

\begin{proof}
As we will shortly observe in Section~\ref{sec:simplicial torsors} below, 
there is a canonical isomorphism 
\begin{equation}
\label{pullback iso}
G\times_H G\cong G\times \ker(\phi). 
\end{equation}
Given a retraction $r\colon G\to \ker(\phi)$, define a map 
$\hat{s}\colon G\to G$ by the composite 
\[
G\xrightarrow{(1,r^{-1})} G\times \ker(\phi)\xrightarrow{1\times i} 
G\times G\stackrel{m}{\to} G. 
\]
Note that $\phi\hat{s} = \phi$.  Note also that 
$\hat{s}$ is $\ker(\phi)$ invariant in the 
sense that the diagram 
\[
\xymatrix{ 
G\times\ker(\phi) \ar[d]_-{p_1} \ar[r]^-m & G \ar[d]^-{\hat{s}} \\ 
G \ar[r]_-{\hat{s}} & G }
\]
in $\cD$ commutes.  From the isomorphism~\eqref{pullback iso}  
above it follows that 
$\hat{s}p_1 = \hat{s}p_2$, 
where $p_1$ and $p_2$ denote the canonical projections 
in the coequalizer diagram 
\[
\xymatrix{G\times_H G \ar@<1ex>[r]^-{p_1} \ar@<-1ex>[r]_-{p_2} & G 
\ar[r]^-{\phi} & H. } 
\]
Therefore there is a unique map $s\colon H\to G$ such that $s\phi = \hat{s}$.  
Hence $\phi s\phi = \phi$ from which it follows that $\phi s = \mathrm{id}$, i.e.\ 
$s$ is a section of $\phi$.  
\end{proof}   
To conclude the proof we observe that since pullbacks and colimits in $\cC_{/B}$ are computed in $\cC$, 
if the underlying map $G_n\to M_{\partial\Delta^n}G$ is an effective epimorphism 
in $\cC$ then the map $G_n\to M_{\partial\Delta^n}G$ is an effective epimorphism 
in $\cC_{/B}$.  
\end{proof} 

\section{Non-abelian cohomology} 
\label{sec:nonabelian}

In this section we study some examples of the internal local homotopy theory 
of the previous section, before we 
recall the definition and some properties of non-abelian 
cohomology.  Throughout this section 
$\cC$ denotes a category equipped with a Grothendieck pretopology, if we need to 
we will suppose that there exists a small subcategory $\cC_0$ of $\cC$ satisfying 
the condition described in Section~\ref{sec:groth tops}.  We begin with 
a discussion of simplicial torsors.  

\subsection{Simplicial torsors} 
\label{sec:simplicial torsors}
One source of examples of local Kan fibrations comes from 
torsors in $s\cC$.  These are studied in detail in the monograph \cite{Duskin-Mem} of Duskin where they 
are used to give an interpretation of cohomology for monads.  
Recall that if $G$ is a group in $s\cC$ and $X$ is an 
object of $s\cC$, then 
a {\em $G$-torsor} in $s\cC$ over $X$ consists of an object 
$P$ of $s\cC/X$ equipped with an action $P\times G\to P$ 
of the group object $X\times G$ in $s\cC/X$ such that the diagram    
\begin{equation} 
\label{pullback diagram for a torsor}
\begin{xy} 
(-8,7.5)*+{P\times G}="1"; 
(-8,-7.5)*+{P}="2"; 
(8,7.5)*+{P}="3"; 
(8,-7.5)*+{X}="4"; 
{\ar "1";"3"}; 
{\ar_-{p_1} "1";"2"}; 
{\ar "2";"4"}; 
{\ar "3";"4"};
\end{xy}
\end{equation}
is a pullback in $s\cC$.  Note that this definition can be expressed entirely 
in terms of finite limits in $s\cC$.  Of course the notion of $G$-torsor over an object $X$ makes sense 
in any category $\cE$ with finite limits for $G$ a group in $\cE$ --- in particular it follows that for $P$ 
and $G$ as above, $P_n$ is a $G_n$ torsor in $\cC$ over $M_n$ for all $n\geq 0$.  

A recurring example of a torsor in this paper is the following: if $\phi\colon G\to H$ is a homomorphism of group objects 
in $s\cC$, then $G$ is a $\ker(\phi)$ torsor over $H$ in $s\cC$.  One way to see that the diagram~\eqref{pullback diagram for a torsor} is a pullback 
in this case is to write it as the composite of the two pullback squares 
\[
\xymatrix{ 
G\times \ker(\phi)\ar[d] \ar[r] & G\times G \ar[d] \ar[r] & G \ar[d] \\ 
G \ar[r] & G\times H \ar[r] & H } 
\]
where the left square is a pullback by definition of $\ker(\phi)$ 
and where the horizontal maps in the right square 
are given by multiplication in $G$ and the composite of 
$\phi\times 1$ and multiplication in $H$ respectively.  
 
We will say that a $G$-torsor $P\to X$ in $s\cC_{/B}$ is {\em locally 
trivial} if each map $P_n\to X_n$ is a local epimorphism in $\cC_{/B}$.   
When $\cC = \cK$ then the notion of locally trivial $G$-torsor for a 
group object in $s\cK_{/B}$ reduces to the notion of 
simplicial parametrized principal bundle in $\cK_{/B}$.  Recall (see 
for example \cite{RS}, Definition 13) that a simplicial parametrized principal bundle in 
$s\cK_{/B}$ consists of a map 
$P\to X$ in $s\cK_{/B}$ together with an action $P\times G\to P$ of $P$ on 
$G$ such that each map $P_n\to X_n$ inherits the structure 
of a parametrized principal $G_n$-bundle in $\cK_{/B}$ and all 
face and degeneracy maps are morphisms of parametrized principal 
bundles.

We have the following result, directly analogous to the 
familiar result in the category of simplicial sets which states that 
every principal bundle in $s\Set$ is a Kan fibration.  
\begin{lemma}
\label{lem:torsors are local fibns}
Let $G$ be a group in $s\cC_{/B}$ and let $P\to X$ be a locally trivial $G$-torsor over $X$ in $s\cC_{/B}$ .  Then $P\to X$ is a 
local Kan fibration in $s\cC_{/B}$.  Moreover, if each $P_n\to X_n$ admits a global section, the 
maps 
\[
P_n \to M_{\Lambda^n_k}P\times_{M_{\Lambda^n_k}X} X_n
\]
admit global sections in $\cC_{/B}$.   
\end{lemma}  
\begin{proof}
  We need to show that the maps 
\[
h^n_k\colon P_n\to M_{\Lambda^n_k}P\times_{M_{\Lambda^n_k}X} X_n 
\]
are local epimorphisms in $\cC_{/B}$. 
Since the functor $M_{\Lambda^n_k}(-)$ is a right adjoint, it sends 
torsors in $s\cC_{/B}$ to torsors in $\cC_{/B}$.  In particular, $M_{\Lambda^n_k}P\to 
M_{\Lambda^n_k}X$ is a $M_{\Lambda^n_k}G$-torsor in $\cC_{/B}$.   Hence 
$M_{\Lambda^n_k}P\times_{M_{\Lambda^n_k}X} X_n\to X_n$ 
is a $M_{\Lambda^n_k}G$-torsor in $\cC_{/B}$.  For convenience of 
notation write $Y_{n,k} =  M_{\Lambda^n_k}P\times_{M_{\Lambda^n_k}X} X_n$.

Let
$(X_{n,i})_{i\in I}$ be a covering family of $X_n$ and let  
$\sigma_i$ be a section of $X_{n,i}\times_{X_n} P_n\to X_{n,i}$ 
for every $i\in I$.  For every $i\in I$ we have a 
commutative diagram 
\[
\xymatrix{ 
X_{n,i}\times_{X_n} P_n \ar[d] \ar[r] & P_n \ar[d] \\ 
X_{n,i}\times_{X_n} Y_{n,k} \ar[d] \ar[r] & Y_{n,k} \ar[d]\\ 
X_{n,i} \ar[r] & X_n } 
\]
in which each square is a pullback.  The section $\sigma_i$ induces a map 
$X_{n,i}\times_{X_n}Y_{n,k}\to P_n$ such that the projection to $Y_{n,k}$ 
differs from the canonical projection $X_{n,i}\times_{X_n}Y_{n,k}\to Y_{n,k}$ 
by a map $X_{n,i}\times_{X_n}Y_{n,k}\to M_{\Lambda^n_k}G$.  By Lemma~\ref{lem:grp} 
there exists a lift $X_{n,i}\times_{X_n}Y_{n,k}\to G_n$ which can be used to 
rescale the induced map $X_{n,i}\times_{X_n}Y_{n,k}\to P_n$ 
to obtain a diagonal filler in the diagram 
\[
\xymatrix{ 
X_{n,i}\times_{X_n} P_n \ar[d] \ar[r] & P_n \ar[d] \\ 
X_{n,i}\times_{X_n} Y_{n,k} \ar@{.>}[ur] \ar[r] & Y_{n,k} } 
\]
for every $i\in I$.  These maps, together with the induced covering 
family $(X_{n,i}\times_{X_n} Y_{n,k}\to Y_{n,k})_{i\in I}$, exhibit 
the maps $h^n_k$ as local epimorphisms in $\cC_{/B}$.   
\end{proof}

\subsection{Twisted cartesian products and the $\overline{W}$ construction} 

Suppose that $P$ is a $G$-torsor over $M$ 
and $U\colon \cC\to \cD$ is a functor.  Following 
\cite{Duskin-Mem}, we say that $P$ is a 
{\em $G$-torsor over $M$} {\em rel.\ }  $U\colon \cC\to \cD$ when 
$P\to M$ is equipped with a section $s\colon U(M)\to U(P)$.  This 
notion allows for a neat definition of the notion of principal twisted 
cartesian product in the category $s\cC$ of simplicial objects of $\cC$, where $\cC$ admits finite limits.  

\begin{definition}[Duskin \cite{Duskin-Mem}] 
Let $G$ be a group object in $s\cC$.  A {\em principal twisted cartesian product} in $s\cC$ with structure group $G$ 
is a $G$-torsor rel.\ $\Dec_0\colon s\cC\to s_a\cC$, where $s_a\cC$ denotes the category of augmented 
simplicial objects in $\cC$ and coherent maps.  
\end{definition} 

Thus a principal twisted cartesian product in $s\cC$ with 
structure group $G$ consists of a $G$-torsor $P\to M$ in $s\cC$ which 
is equipped with a {\em pseudo-cross section}, i.e.\ for 
every $n\geq 0$ there is a map $\sigma_n\colon M_n\to P_n$ such that 
$\sigma_n$ is a section of $P_n\to M_n$ for all $n\geq 0$, 
which satisfies $d_i\sigma_n = \sigma_{n-1}d_i$ for all $0<i\leq n$ 
and $s_i\sigma_{n-1} = \sigma_{n}s_i$ for all $0\leq i\leq n$, 
$n\geq 1$ (in other words this reduces to the classicial definition --- see Definition 18.5 of \cite{MaySOAT}).   
As an immediate application of Lemma~\ref{lem:torsors are local fibns} we have the following 
result. 
\begin{lemma} 
\label{lem:PTCPs are local fibns}
Suppose that $P\to M$ is a principal twisted cartesian product in $s\cC_{/B}$ with 
structure group $G$.  Then $\pi\colon P\to M$ is a local fibration in $s\cC_{/B}$.  
\end{lemma} 
 
A prime example of a twisted cartesian product is the 
universal $G$ torsor $WG\to \overline{W}G$.  Recall (see \cite{RS}) that if $G$ is a group in 
$s\cC$, then there is a canonical $G$-torsor 
\[
G\dslash G \to 1\dslash G
\]
in $s\Gpd(\cC)$, where $\Gpd(\cC)$ denotes the category of groupoids 
in $\cC$.  Here $G\dslash G$ denotes the action 
groupoid associated to the action of the group $G$ in 
$\Gpd(\cC)$ on itself by 
right multiplication, likewise $1\dslash G$ denotes the action 
groupoid of the trivial action of the group $G$ on the 
terminal groupoid 1 in $\Gpd(\cC)$.  
Applying the nerve functor $N$ gives a canonical 
$G$-torsor 
\[
N(G\dslash G)\to NG 
\]
in $ss\cC$, where $G$ is thought of as a constant simplicial group in $s\cC$, and 
where $NG = N(1\dslash G)$ denotes the nerve of the usual groupoid associated to $G$.  
Applying the total simplicial object functor $\sigma_*\colon ss\cC\to s\cC$ right adjoint to 
$\Dec\colon s\cC\to ss\cC$ (see \cite{CR,Ste1}) we get 
a $G$-torsor 
\[
\sigma_*N(G\dslash G)\to \sigma_*NG 
\]
in $s\cC$.  This $G$-torsor is denoted $WG\to \Wbar G$.  It can be shown (see \cite{RS}) 
that $WG = \Dec_0\Wbar G$ and that when $\cC = \Set$, we recover the usual 
description of the classifying complex $\Wbar G$ of a simplicial group $G$.  Thus in this 
case $\Wbar G$ is the simplicial set with exactly one 0-simplex and whose set of 
$n$-simplices, $n\geq 1$, is the set 
\[
(\Wbar G)_n = G_{n-1}\times \cdots \times G_0 
\]
with face and degeneracy maps defined by the formulas 
\[
d_i(g_{n-1},\ldots, g_0) = \begin{cases}
		(g_{n-2},\ldots, g_0) & \text{if}\ i=0, \\
		(d_i(g_{n-1}),\ldots, d_1(g_{n-i+1}),g_{n-i-1}d_0(g_{n-i}),g_{n-i-2},\ldots, g_0) & \text{if}\ 1\leq i\leq n
		\end{cases}
\]
and 
\[
s_i(g_{n-1},\ldots, g_0) = \begin{cases} 
			(1,g_{n-1},\ldots, g_0) & \text{if}\ i=0, \\ 
			(s_{i-1}(g_{n-1}),\ldots, s_0(g_{n-i}),1,g_{n-i-1},\ldots, g_0) & \text{if}\ 1\leq i\leq n.  
			\end{cases}
\]
respectively.  When $G$ is a group in $s\cC$ whose underlying simplicial object 
is constant, one can show that $\Wbar G$ reduces to the usual classifying space 
$BG$ of $G$, i.e.\ the simplicial object in $\cC$ whose object of $n$-simplices 
is the product $G\times \cdots \times G$ ($n$ factors) and whose face and degeneracy 
maps are defined using the product in $G$ and insertion of identities.

For the case of simplicial groups it is well known that 
$\overline{W}G$ is a fibrant simplicial set --- see for instance Lemma 21.3 of \cite{MaySOAT}.  
In the next lemma we prove a more general statement.  

\begin{lemma} 
Let $G$ be group in $s\cC_{/B}$.  Then $\overline{W}G$ is a locally fibrant simplicial object in $\cC_{/B}$.  
\end{lemma} 

One way to prove this is to use a Yoneda argument, as in the proof of Lemma~\ref{lem:grp}.  
We give here a different proof, which recasts the argument from Lemma 21.3 of \cite{MaySOAT} 
in a modern framework.  

\begin{proof} 
We need to prove that $(\W G)_n\to M_{\Lambda^n_k}\W G$ is a local epimorphism for all $0\leq k\leq n$ and all $n\geq 1$.  Observe that if 
$k<n$ then 
\[
\Lambda^n_k = C\Lambda^{n-1}_{k}\cup_{\Lambda^{n-1}_{k}} \Delta^{n-1}  
\]
(see for instance Lemma 3.11 of \cite{Joyal}). 
Using this description of $\Lambda^n_k$ and Corollary~\ref{cor:cones and matching objects}, we have   
\begin{align*} 
M_{\Lambda^n_k}\W G & = M_{C\Lambda^{n-1}_{k}\cup_{\Lambda^{n-1}_{k}} \Delta^{n-1}} \W G \\ 
& = M_{C\Lambda^{n-1}_{k}} \W G\times_{M_{\Lambda^{n-1}_{k}}\W G} (\W G)_{n-1} \\ 
& = M_{\Lambda^{n-1}_{k}} W G\times_{M_{\Lambda^{n-1}_{k}}\W G} (\W G)_{n-1}.  
\end{align*}
Therefore the map $(\W G)_n\to M_{\Lambda^n_k}\W G$ is isomorphic to the map 
\[
(W G)_{n-1} \to M_{\Lambda^{n-1}_{k}} W G\times_{M_{\Lambda^{n-1}_{k}}\W G} (\W G)_{n-1}
\]  
and so in this case the result follows from Lemma~\ref{lem:torsors are local fibns}.  
For the case where $k=n$, note that we have 
\[
M_{\Lambda^n_n}\W G = M_{\Lambda^n_0} \W G^o,   
\]
where $G^o$ denotes the opposite simplicial object to $G$, so that $G^o$ is the 
composite of the functor $G\colon \Delta^\op\to \cC$ and the functor $\tau^\op$, 
where $\tau\colon \Delta\to \Delta$ is the automorphism of $\Delta$ which reverses 
the order of each ordinal.  
Therefore the map $(\W G)_n \to M_{\Lambda^n_n}\W G$ is isomorphic to the map 
$(\W G^o)_n \to M_{\Lambda^n_0} \W G^o$ which we have just seen is a local epimorphism.  
\end{proof}

%
%

\subsection{Non abelian cohomology} 
\label{non abelian cohomology}
 
As mentioned in the introduction, the \v{C}ech 
cohomology $\check{H}^1(B,G)$ 
of a topological space $B$ with coefficients 
in a presheaf of groups $G$ on $B$ is 
traditionally thought of in terms of cocycles 
$g_{ij}\in G(U_i\cap 
U_j)$ 
relative to some open cover $U = (U_i)_{i\in I}$ of $B$.  
It is often more convenient however to package 
the data of such a cocycle $g_{ij}$ into a simplicial map 
\[ 
\check{C}(U)\to \Wbar G 
\] 
from the \v{C}ech resolution $\check{C}(U)$ of 
the open cover $U$ to the classifying simplicial 
presheaf $\Wbar G$.
The \v{C}ech resolution $\check{C}(U)$ is the 
simplicial presheaf on $B$ defined as follows.  
Let $rU_i$ denote the presheaf on $\Open(B)$ represented by 
$U_i$, so that 
\[
rU_i(V) = \begin{cases} 
1 & \text{if}\ V\subset U_i, \\ 
\emptyset & \text{otherwise}
\end{cases} 
\] 
for $V$ an open subset of $B$.  
Let $U$ also denote the coproduct $U = \coprod_{i\in I}rU_i$.  
Then $\check{C}(U)$ is the simplicial presheaf whose presheaf of $n$-simplices is 
\[ 
[n]\mapsto U\times U\times \cdots \times U\ (n+1\ \text{factors}).  
\] 
Thus the presheaf of 2-simplices of $\check{C}(U)$ 
is the coproduct $\coprod_{i,j\in I}r(U_i\cap U_j)$ 
and so on.  It is not hard to see that a morphism of 
simplicial presheaves $\check{C}(U)\to \Wbar G$ 
corresponds to the data 
of a family of sections $g_{ij}\in G(U_i\cap U_j)$ 
satisfying the cocycle equation $g_{ij}g_{jk} = g_{ik}$.  
We then have the 
following description of the \v{C}ech cohomology set 
$\check{H}^1(B,G)$ as the colimit 
\[
\check{H}^1(B,G) =  \varinjlim_{U\in \Cov(B)} \pi(\check{C}(U),\Wbar G)
\] 
of the sets $\pi(\check{C}(U),\Wbar G)$ of simplicial homotopy classes of maps 
over the set $\Cov(B)$ of all covers $U$ of $B$.  
The 
canonical map $\check{C}(U)\to 1$ to the terminal simplicial 
sheaf  on $B$ is a hypercover and so we have a canonical map 
\[
\check{H}^1(B,G) \to \varinjlim_{V\in \pi HR(B)} \pi(V,\Wbar G).  
\]
It turns out that this map is an isomorphism.  
This fact was first observed in \cite{Jardine3} and (at this level of generality) rests on the 
fact that for any hypercover $V$, there is a bijection 
$\pi(V,\Wbar G) = \pi(\cosk_0 V,\Wbar G)$.  To prove this it is 
sufficient to check that the canonical map 
$V\to \cosk_0V$ is an isomorphism on fundamental groupoids, and by 
passing to stalks it is sufficient to check this for a contractible Kan complex, which is straightforward.  
In \cite{Jardine3} Jardine observes that one obtains an identification 
\[ 
\check{H}^1(B,G) \cong [1,\Wbar G] 
\] 
on applying the generalized Verdier hypercovering 
theorem (Theorem~\ref{verdier hypercovering theorem}) --- in fact 
Jardine proves that this identification holds more generally for arbitrary Grothendieck 
topoi.  
In other words, \v{C}ech cohomology and so-called 
hyper-\v{C}ech cohomology coincide in this case.  
This identification opens the way to define higher 
order generalizations of the \v{C}ech cohomology $\check{H}^1(B,G)$.  
The following definition can be found in \cite{Breen1} for the 
case when $G$ is 2-truncated (see also \cite{Noohi} for a nice discussion 
in the context of group cohomology).  

\begin{definition}[\cite{Breen1}] 
Let $B$ be a topological space and let $G$ be a 
presheaf of simplicial groups on $B$.  Then we define 
\[ 
H(B,G) := [1,\overline{W}G] 
\] 
where the square brackets $[\, ,\, ]$ denotes the set of 
morphisms in the homotopy category $\Ho(s\Pre(B)_\mathcal{L})$.  
\end{definition} 

We will call $H(B,G)$ the {\em hyper-\v{C}ech cohomology} of $B$ with coefficients in $G$.  
Elements of $H(B,G)$ have the following explicit description using Jardine's notion 
of a {\em cocycle category} \cite{Jardine2,Jardine4}.  A {\em $G$-cocycle} on 
$B$ consists of a locally fibrant, locally acyclic simplicial object $V$ in $\cK_{/B}$, together 
with a simplicial map $V\to \Wbar G$.  A morphism of $G$-cocycles from $p\colon V\to \Wbar G$ 
to $q\colon W\to \Wbar G$ consists of a map $f\colon V\to W$ in $\cK_{/B}$ such that 
$q f = p$.  There is an obvious notion of composition of morphisms of $G$-cocycles so 
that $G$-cocycles form the objects of a category $\bH(B,\Wbar G)$.  It can be shown 
(using Theorem 3 of \cite{Jardine2}) that there is an isomorphism $\pi_0 \bH(B,\Wbar G) = H(B,G)$.       
   
When $G$ is the presheaf of constant simplicial groups associated to a 
presheaf of groups on $B$, the cohomology set $H(B,G)$ reduces to the familiar 
\v{C}ech cohomology set $\check{H}^1(B,G)$.  This set $\check{H}^1(B,G)$ is of course a shadow of 
a richer structure, in the sense that $\check{H}^1(B,G)$ is the set of path components of the groupoid of $G$-torsors on $B$.  
Similarly, $H(B,G)$ is also a shadow of a richer structure, in this case $H(B,G)$ is the set 
of path components of the $\infty$-groupoid of $G$-torsors on $B$.  This $\infty$-groupoid 
can be explicitly described as the Kan complex $\hom(1,R\Wbar G)$, where $R\Wbar G$ is a 
fibrant replacement (i.e.\ (hyper-)$\infty$-stackification) for $\Wbar G$, and where 
$\hom(-,-)$ denotes the simplicial enrichment for $s\Pre(B)$.        

One can define \v{C}ech versions $\check{H}(B,G)$ of the sets $H(B,G)$ by setting 
\[ 
\check{H}(B,G) = \varinjlim_{U\in \Cov(B)} \pi(\check{C}(U),\W G).  
\] 
As above we still have a canonical map 
\[ 
\check{H}(B,G)\to H(B,G)  
\] 
which is essentially the comparison map between \v{C}ech 
and sheaf cohomology.  It is well known that this map need not be an isomorphism.  
However in the special case when $B$ is paracompact and $G$ is $k$-truncated for some $k\geq 0$, in the sense 
that the sheaves of homotopy groups $\pi_i(G)$ are trivial for $i\geq k$, then it can be shown (using 
Lemma 7.2.3.5 of \cite{Lurie} for instance) that this comparison map is an isomorphism.  

An interesting case of this theory is the following.  
Suppose that $G$ is a presheaf of groups on $B$.  
Then $G$ determines a group object $\AUT(G)$ in the 
category of presheaves of groupoids on $B$; briefly 
$\AUT(G)$ is the action groupoid in the category of group objects 
in $\Pre(B)$ associated to the homomorphism of 
groups $G\to \Aut(G)$.  
Thus $\AUT(G)$ is an example of a {\em 2-group} or {\em categorical group} \cite{Baez-Lauda}.  
Applying the nerve construction $N$ to the 
presheaf of groupoids $\AUT(G)$ yields a group in $s\Pre(M)$ 
that we will also denote by $\AUT(G)$.  
We will write $H^1(B,\AUT(G))$ for the corresponding cohomology set.  

The set $H^1(B,\AUT(G))$ parametrizes the set of equivalence classes of {\em $G$-gerbes} on $B$.  
A $G$-gerbe on $B$ is a generalization of the notion of $G$-torsor on $B$.  For more details we refer to 
\cite{Breen2}.  In this case the group $\AUT(G)$ is 2-truncated so that if $B$ is paracompact 
we have an isomorphism $\check{H}^1(B,\AUT(G)) = H^1(B,\AUT(G))$.  

When $G$ is the presheaf of groups on $B$ represented a group $G$ in $\cK$, there is a slight variant of this 
construction in which the sheaf of groups $\Aut(G)$ is replaced by the sheaf represented by the 
group $\Aut(G)$ in $\cK$.  In this case one obtains a group $\AUT_0(G)$ in $s\Sh(B)$, represented 
by a group in $s\cK_{/B}$, and one defines as above the set 
$H^1(B,\AUT_0(G))$.  Note that there is a canonical homomorphism $\AUT_0(G)\to \AUT(G)$ but 
the induced map $H^1(B,\AUT_0(G))\to H^1(B,\AUT(G))$ need not be an isomorphism.  

\section{Proof of the main result} 
\label{sec:proof of main thm} 

Let us recall the statement of the main theorem of this paper.  
\setcounter{section}{1}
\setcounter{theorem}{0}
\begin{theorem}
Let $B$ be a paracompact, Hausdorff space and let $G$ 
be a fibrant simplicial parametrized group in $\cK_{/B}$.  
If $G_n$ is well sectioned for all $n\geq 0$ then there is an isomorphism 
of sets 
\[ 
H(B,G) = \check{H}^1(B,|G|).  
\]
\end{theorem} 
\setcounter{section}{5}
\setcounter{theorem}{0}

In this section we will give a proof of this theorem.  We begin by describing a 
universal $G$-cocycle.  

\subsection{The universal $G$-cocycle} 
\label{sec:univ cocycle}
Let $G$ be a simplicial group in $\cK_{/B}$.  If we apply Illusie's total d\'{e}calage functor~\eqref{total dec} 
to $G$ we obtain a biaugmented bisimplicial group $\Dec(G)$ in $\cK_{/B}$.  The 
bisimplicial group $\Dec(G)$ plays a prominent role in Porter's work \cite{Porter} 
on homotopy $n$-types.    
We will think of $\Dec(G)$ as the simplicial simplicial group in $\cK_{/B}$ whose 
simplicial group of $p$-simplices in $s\cK_{/B}$ is $\Dec_{p}(G)$ (i.e.\  the $p$-th row of $\Dec(G)$).  
It is helpful to keep the following picture in mind: 
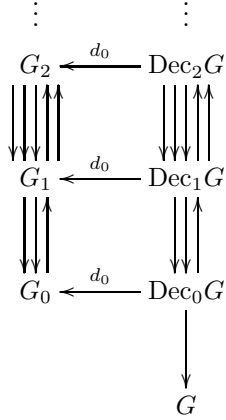
\[ 
\xy 
(0,0)*+{G_0}="1"; 
(20,0)*+{\Dec_0G}="2"; 
(20,-15)*+{G}="3"; 
(0,15)*+{G_1}="4"; 
(20,15)*+{\Dec_1 G}="5"; 
(0,30)*+{G_2}="6"; 
(20,30)*+{\Dec_2G}="7"; 
(0,38)*+{\vdots}; 
(20,38)*+{\vdots};
{\ar_-{d_0} "2";"1"};
{\ar_-{d_0} "5";"4"}; 
{\ar_-{d_0} "7";"6"};
{\ar "2";"3"}; 
{\ar@<-1ex> "4";"1"}; 
{\ar "4";"1"}; 
{\ar@<-1ex> "1";"4"}; 
{\ar@<-1ex> "5";"2"}; 
{\ar "5";"2"}; 
{\ar@<-1ex> "2";"5"};
{\ar@<-2ex> "6";"4"}; 
{\ar@<-1ex> "6";"4"}; 
{\ar "6";"4"}; 
{\ar@<-1ex> "4";"6"}; 
{\ar@<-2ex> "4";"6"}; 
{\ar@<-2ex> "7";"5"}; 
{\ar@<-1ex> "7";"5"}; 
{\ar "7";"5"}; 
{\ar@<-1ex> "5";"7"}; 
{\ar@<-2ex> "5";"7"};
\endxy
\]
We can take fiberwise geometric realizations in the horizontal direction to     
get the following diagram of 
simplicial groups in $\cK_{/B}$ : 
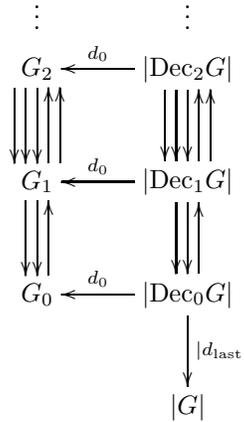
\[
\xy 
(0,0)*+{G_0}="1"; 
(20,0)*+{|\Dec_0G|}="2"; 
(20,-15)*+{|G|}="3"; 
(0,15)*+{G_1}="4"; 
(20,15)*+{|\Dec_1G|}="5"; 
(0,30)*+{G_2}="6"; 
(20,30)*+{|\Dec_2G|}="7"; 
(0,38)*+{\vdots}; 
(20,38)*+{\vdots};
{\ar_-{d_0} "2";"1"};
{\ar_-{d_0} "5";"4"}; 
{\ar_-{d_0} "7";"6"};
{\ar^-{|d_\last|} "2";"3"}; 
{\ar@<-1ex> "4";"1"}; 
{\ar "4";"1"}; 
{\ar@<-1ex> "1";"4"}; 
{\ar@<-1ex> "5";"2"}; 
{\ar "5";"2"}; 
{\ar@<-1ex> "2";"5"};
{\ar@<-2ex> "6";"4"}; 
{\ar@<-1ex> "6";"4"}; 
{\ar "6";"4"}; 
{\ar@<-1ex> "4";"6"}; 
{\ar@<-2ex> "4";"6"}; 
{\ar@<-2ex> "7";"5"}; 
{\ar@<-1ex> "7";"5"}; 
{\ar "7";"5"}; 
{\ar@<-1ex> "5";"7"}; 
{\ar@<-2ex> "5";"7"};
\endxy
\]
We can apply the functor $\overline{W}$ to this diagram and obtain the following diagram of simplicial spaces  
\begin{equation} 
\label{double fibration}
\xy 
(0,7.5)*+{\overline{W}G}="1"; 
(20,7.5)*+{\overline{W}|\Dec(G)|}="2"; 
(20,-7.5)*+{\overline{W}|G|}="3"; 
{\ar "2";"1"}; 
{\ar "2";"3"};
\endxy
\end{equation} 
We can now explain the main idea of our proof of Theorem~\ref{main result}.  The diagram~\eqref{double fibration} 
induces a diagram 
\[
\xymatrix{ 
H(B,G) & H(B,|\Dec\, G|) \ar[l] \ar[d] \\ 
& H(B,|G|) } 
\]
so that we can compare $H(B,G)$ and $H(B,|G|)$ via $H(B,|\Dec\, G|)$.  
We will prove that each map in this diagram is an isomorphism using the hypotheses on $G$ and $B$.  
We remark that it is easy to check that if $G$ is an abelian group object in $s\cK_{/B}$ then the maps 
in the diagram above are homomorphisms of abelian groups.  
  
To prove that the maps in the above diagram are isomorphisms 
we will prove the following two propositions.  After we had completed this paper, 
we discovered that a result related to, but less general than Proposition~\ref{hypercover} below 
had been proven earlier by Moerdijk --- see 3.7 of \cite{Moerdijk}.  
   
\begin{proposition} 
\label{hypercover}
Let $G$ be a well sectioned group in $s\cK_{/B}$.  Then the map 
\[
\overline{W}|\Dec\, G|\to \overline{W}|G|\]
is a locally acyclic local fibration in $s\cK_{/B}$.  
\end{proposition} 

It follows that the induced map $\overline{W}|\Dec\, G|\to \overline{W}|G|$ between representable simplicial presheaves 
on $B$ is a locally acyclic local fibration in $s\Pre(B)$, and therefore in particular is a weak equivalence 
in the localized projective model structure on $s\Pre(B)$.  Therefore the map 
\[
H(B,|\Dec\, G|)\to H(B,|G|) 
\]
is an isomorphism.  Theorem~\ref{main result} then follows from the next proposition.    

\begin{proposition} 
\label{2nd prop in proof of main thm}
Suppose that $B$ is paracompact and that $G$ is a well sectioned group in $s\cK_{/B}$.  
Then the induced map 
\[
H(B,|\Dec\, G|)\to H(B,G)
\]
is an isomorphism.  
\end{proposition}

Observe that the diagram~\eqref{double fibration} gives rise to a $G$-cocycle on 
the classifying space $B|G|$ using the canonical map $\check{C}(E|G|)\to \Wbar |G|$.  
We call this $G$-cocycle the {\em universal $G$-cocycle} for the following reason: any 
$G$-cocycle over $B$ is equivalent to a $G$-cocycle obtained by pulling back the 
universal $G$-cocycle along a map $B\to B|G|$, unique up to homotopy.  To see this observe 
that the isomorphism $\check{H}^1(M,|G|)\cong H(M,G)$ is realized by the map 
which sends a $|G|$-cocycle on $B$ determined by a map $\check{C}(P)\to \Wbar |G|$, where $P\to B$ is a 
principal $G$-bundle in $\cK_{/B}$, to the map $X\to \Wbar G$, where $X$ is defined by the pullback diagram 
\[
\xymatrix{ 
X \ar[d] \ar[r] & \Wbar |\Dec\, G| \ar[d] \\ 
\check{C}(P) \ar[r] & \Wbar |G|, } 
\]
and $X\to \Wbar G$ is the canonical map induced by $\Wbar |\Dec\, G|\to \Wbar G$.  
Since the map $\check{C}(P)\to \Wbar |G|$ factors through $\check{C}(E|G|)$ via a classifying map 
$P\to E|G|$ for the bundle $P$, it follows 
that the cocycle $X\to \Wbar G$ is the pullback of the universal $G$-cocycle under the 
map $B\to B|G|$.  Theorem~\ref{main result} implies that any $G$-cocycle on $B$ is equivalent 
to a cocycle of this form.

\subsection{Proof of Proposition~\ref{hypercover}}
\label{Proof of first Proposition}
To prove Proposition~\ref{hypercover} it suffices to prove that $|\Dec\, G|\to |G|$ is a locally acyclic local fibration, since 
$\Wbar$ preserves surjective local fibrations and local weak equivalences (it suffices to check this on stalks, and it is 
well known that the analogous statements are 
true for simplicial sets).  Therefore, in view of Proposition~\ref{contractibility propn}, 
it suffices to prove three things: 
\begin{enumerate} 
\item the map $|\Dec_0G|\to |G|$ admits local sections, 

\item the map $|\Dec_1 G|\to |\Dec_0G|\times_{|G|}|\Dec_0G|$ admits local sections, 

\item the map $|\Dec_nG|\to M_{\partial\Delta^n}|\Dec\, G|$ is an effective epimorphism for all $n\geq 2$.  
\end{enumerate} 

The main technical tool that we will use in proving these results is the following theorem from \cite{RS}.  
\begin{theorem}[\cite{RS}] 
\label{main thm from RS}
Let $G$ be a fibrant parametrized simplicial group.  Suppose that $P\to M$ is a 
simpliciail parameterized principal $G$-bundle, so that 
$P_n\to M_n$ is a numerable parametrized principal $G_n$ bundle in $\cK_{/B}$ for all $n\geq 0$.  If $M$ is a 
proper simplicial object, then the induced map 
\[
|P|\to |M| 
\]
on fiberwise geometric realizations is a locally trivial parametrized principal $|G|$ bundle in $\cK_{/B}$.  
\end{theorem} 
Recall that a simplicial object $X$ in $s\cK_{/B}$ is said to be {\em proper} if the inclusion $L_nX\subset X_{n+1}$ 
is an $\bar{f}$-cofibration in $\cK_{/B}$ for all $n\geq 0$.  In other words $X$ is proper if the inclusion 
$(X_{n+1},sX_n)$ is a fiberwise NDR pair for all $n\geq 0$ where $sX_n = \cup^n_{i=0}s_iX_n$.

We need a recognition theorem to identify when a simplicial object in $\cK_{/B}$ is proper.  It is sometimes 
easier to recognize when a simplicial object $X$ in $\cK_{/B}$ is {\em good} in the sense that $s_i\colon X_n\to X_{n+1}$ is a 
$\bar{f}$-cofibration in $\cK_{/B}$ for all $0\leq i\leq n$ and all $n\geq 0$.  The following result 
from \cite{RS} (which is essentially folklore) provides a useful recognition criterion.  
\begin{proposition}[\cite{RS}] 
\label{good implies proper}
If $X$ is a good simplicial object in $\cK_{/B}$, then $X$ is proper.  If $G$ is a well sectioned group in $s\cK_{/B}$, 
then $G$ is good and hence proper.  
\end{proposition}

\begin{lemma}
\label{geometric realization of Dec G to G}
Let $G$ be a well sectioned, fibrant simplicial parametrized group.  Then the 
fiberwise geometric realization 
\[
|d_{\last}|\colon |\Dec_0G|\to |G| 
\]
is the projection in a 
numerable, locally trivial parametrized principal $|\ker(d_{\last})|$ bundle on 
$|G|$.  In particular $|d_{\last}|$ admits local sections.  
\end{lemma} 
\begin{proof}
The short exact sequence of groups $1\to \ker(d_\last)\to \Dec_0 G\to G\to 1$ in $s\cK_{/B}$ makes it clear 
that $d_\last\colon \Dec_0 G\to G$ has the structure of a simplicial parametrized 
principal $\ker(d_\last)$-bundle over $G$: the action of 
$\ker(d_\last)$ on $\Dec_0G$ is the obvious one and the diagram 
\[
\xymatrix{ 
\Dec_0 G\times \ker(d_\last) \ar[r] \ar[d] & \Dec_0 G \ar[d] \\ 
\Dec_0 G \ar[r] & G } 
\]
is clearly a pullback in $s\cK_{/B}$.  We need to check that (i) each $(\Dec_0G)_n\to G_n$ is a numerable $\ker(d_{n+1})$ bundle 
in $\cK_{/B}$, and that (ii) $\ker(d_{n+1})$ is a fibrant goup in $\cK_{/B}$.  The first statement is clear, since the map 
$d_{n+1}\colon (\Dec_0G)_n\to G_n$ has a section, and hence is a trivial bundle.  For the second statement, 
we observe that $\ker(d_{n+1})$ is a retract of $G_n$ in $\cK_{/B}$, and hence is fibrant since $G_n$ is.      
We now wish to apply Theorem~\ref{main thm from RS} to conclude that $|\Dec_0 G|\to |G|$ is a locally trivial 
principal bundle in $\cK_{/B}$, and so we need to know that $G$ is proper.  This follows from Proposition~\ref{good implies proper}: 
$G$ is well sectioned, hence good, and hence proper.    
\end{proof}

We now turn our attention to item 2.  We have the following lemma.  
\begin{lemma} 
Let $G$ be a well sectioned, fibrant group in $s\cK_{/B}$.  Then the fiberwise geometric realization 
\[
|\Dec_1 G|\to |\Dec_0G|\times_{|G|}|\Dec_0G| 
\]
is the projection in a locally trivial parametrized principal bundle in $\cK_{/B}$, and hence admits local sections.  
\end{lemma}

\begin{proof}
We need to verify the hypotheses from Theorem~\ref{main thm from RS}.  We first check that each 
map 
\begin{equation}
\label{2nd bundle projection}
(\Dec_1G)_n \to (\Dec_0G)_n\times_{G_n}(\Dec_0G)_n 
\end{equation}
is the projection in a numerable parametrized principal bundle.  To do this we will show that each of these maps has a section.  
The map~\eqref{2nd bundle projection} is the map 
\begin{gather*} 
G_{n+2}\to G_{n+1}\times_{G_n}G_{n+1} \\ 
g\mapsto (d_{n+1}(g),d_{n+2}(g)).  
\end{gather*}   
Suppose that $(x,y)\in G_{n+1}\times_{G_n}G_{n+1}$ so that  
$d_{n+1}(x) = d_{n+1}(y)$.  Then it is easy to check that 
\[
\sigma(x,y) = s_n(xy^{-1})s_{n+1}(y) 
\]
satisfies $d_{n+1}(\sigma(x,y)) = x$ and $d_{n+2}(\sigma(x,y)) = y$.  Hence~\eqref{2nd bundle projection} is the projection 
in a trivial fiberwise principal bundle.  Next we check that the structure group 
of this bundle is fibrant.  Thus we must check that $\ker(d_{n+1})\cap \ker(d_{n+2})$ is a fibrant 
object of $\cK_{/B}$.  As in the proof of Proposition~\ref{contractibility propn}, one may show that 
$\ker(d_{n+1})\cap \ker(d_{n+2})$ is a retract of $G_n$, and hence is fibrant.      

Finally, we need to check that $\Dec_0G\times_G \Dec_0G$ is a 
proper simplicial object.  Observe that there is an isomorphism 
$\Dec_0G\times_G \Dec_0G \cong \Dec_0G\times_B \ker(d_\last)$.  
It suffices to prove that $\Dec_0G\times_B \ker(d_\last)$ is well sectioned.  
Clearly $\Dec_0G$ is well sectioned; since $\ker(d_\last)$ is a 
retract of $\Dec_0 G$ we see that $\ker(d_\last)$ is also well sectioned.  
Therefore the result follows from the fact that if $A\to X$ and $A'\to X'$ are $\bar{f}$-cofibrations 
in $\cK_{/B}$, then $A\times_B A'\to X\times_B X'$ is also an $\bar{f}$-cofibration (see 
\cite{RS}, Lemma 24).  
\end{proof} 

\begin{lemma} 
The map $|\Dec_n G|\to M_{\partial\Delta^n}|\Dec\, G|$ is an effective epimorphism for all $n\geq 2$.  
\end{lemma} 

\begin{proof} 
Observe that since fiberwise geometric realization preserves finite limits, it 
suffices to prove that 
\[
|\Dec_n G|\to |M_{\partial\Delta^n}\Dec\, G|
\]
is an effective epimorphism in $\cK_{/B}$.  Since fiberwise geometric realization preserves colimits, and 
colimits in $s\cK_{/B}$ are computed pointwise,  
it suffices to prove that 
\begin{equation}
\label{boundary filling in main thm}
(\Dec_n G)_m \to (M_{\partial\Delta^n}\Dec\, G)_m 
\end{equation}
is an effective epimorphism in $\cK_{/B}$ for all $m\geq 0$ and all $n\geq 2$.  
Since colimits in $\cK_{/B}$ are computed in $\cK$ and then equipped with the canonical map 
to $B$, it is sufficient to prove that~\eqref{boundary filling in main thm} is an effective epimorphism 
in $\cK$ for all $m\geq 0$ and all $n\geq 2$.  Since colimits in $\cK$ are computed in $\Top$ it is 
sufficient to prove that the diagram 
\[ 
(\Dec_n G)_m\times_{(M_{\partial\Delta^n}\Dec\, G)_m} (\Dec_n G)_m 
\rightrightarrows (\Dec_n G)_m \to (M_{\partial\Delta^n}\Dec\, G)_m 
\]
is a coequalizer in $\Top$ for all $m\geq 0$ and all $n\geq 2$.  Therefore it is sufficient to prove 
that~\eqref{boundary filling in main thm} is a surjective quotient map for all $m\geq 0$ and all $n\geq 2$.  
Clearly 
\[
(M_{\partial\Delta^n}\Dec\, G)_m = M_{\Delta^m}M_{\partial\Delta^n}\Dec\, G = 
M_{\Delta^m\Box \partial\Delta^n}\Dec\, G, 
\]
using~\eqref{external product and matching objects}.  Therefore, Proposition~\ref{bisimplicial matching and Dec} shows 
that the map~\eqref{boundary filling in main thm} is isomorphic to the map 
\[
G_{n+m+1}\to M_{\Delta^m\star\partial\Delta^n}G, 
\]
where $\Delta^m\star\partial\Delta^n$ denotes the join of the simplicial sets $\Delta^m$ and 
$\partial\Delta^n$.  
The join $\Delta^m\star\partial\Delta^n$ is computed in \cite{Joyal} to be equal to 
\[
\bigcup^n_{j=0} \partial_{j+m+1}\Delta^{n+m+1}. 
\]
It is easy to see that the geometric realization of $\Delta^m\star\partial\Delta^n$ contracts 
onto its final vertex and therefore the inclusion $\Delta^m\star\partial\Delta^n\to \Delta^{n+m+1}$ is an 
acyclic cofibration of simplicial sets.  Therefore the map $G_{n+m+1}\to M_{\Delta^m\star\partial\Delta^n}G$ 
admits local sections by Lemma~\ref{henriques lemma}, and hence is a surjective quotient map.    
\end{proof}  

\subsection{Proof of Proposition~\ref{2nd prop in proof of main thm}}

We now consider the map $\overline{W}|\Dec\, G|\to \overline{W}G$ from~\eqref{double fibration}.  This is induced by 
the homomorphism of groups $|\Dec\, G|\to G$ in $s\cK_{/B}$ which is in turn induced by the homomorphism 
of simplicial groups $d_\first\colon \Dec\, G\to G$ in $s\cK_{/B}$ which in degree $n$ is given by the augmentation homomorphism 
\[
d_\first\colon \Dec_n G\to G_n.   
\]
Since this augmentation homomorphism can be re-written as the augmentation homomorphism 
\[
\Dec_0(\Dec_{n-1}G)\to (\Dec_{n-1}G)_0 
\]
it is enough to study the 
homomorphism $d_0\colon \Dec_0 G\to G_0$ for any group $G$.  For this we have the following lemma.  
\begin{lemma}
\label{summary of dec}
Let $G$ be a well sectioned parametrized simplicial group.  Then the following are true.  
\begin{enumerate} 
\item $\Dec_0G$ is a split augmented group in $s\cK_{/B}$ with augmentation $d_0\colon \Dec_0G\to G_0$ and splitting 
$s_0\colon G_0\to G_1$.  

\item The group $\ker(d_0)$ is well sectioned in $s\cK_{/B}$.  If $G$ is fibrant then so is $\ker(d_0)$.    

\item The short exact sequence of groups in $\cK_{/B}$, 
\[
1\to |\ker(d_0)|\to |\Dec_0G|\to G_0\to 1, 
\]
obtained by applying the fiberwise geometric realization functor $|\cdot |$ is split exact and 
$|\ker(d_0)|$ is fiberwise contractible in $\cK_{/B}$.  

\end{enumerate}
\end{lemma}  

\begin{proof}
The first statement of the lemma is clear.  For the second statement, 
observe that $\ker(d_0)$ is a retract of $G$, and hence is well sectioned 
since $\bar{f}$-cofibrations are closed under retracts.  A similar remark applies to the 
statement regarding fibrancy.    Finally, 
for the last statement, we prove that $|\ker(d_0)|$ is fiberwise contractible 
in $\cK_{/B}$ (the split exactness is clear from 1).  To accomplish this we 
will prove that the simplicial object $\ker(d_0)$ in $\cK_{/B}$ has extra degeneracies.  
In fact, because of the simplicial identity $d_0s_{n+1} = s_{n}d_0$ for $i<j$,
 the left over degeneracy map $s_{n+1}\colon (\Dec_0G)_n\to (\Dec_0G)_{n+1}$ restricts to define a homomorphism 
 \[
 s_{n+1} \colon (\ker(d_0))_n\to (\ker(d_0))_{n+1} 
 \]
 for all $n\geq 0$.  One can therefore define a simplicial homotopy $h\colon \ker(d_0)\otimes \Delta^1\to \ker(d_0)$ 
 which fits into a commutative diagram 
 \[
 \xymatrix{ 
 \ker(d_0) \ar[d] \ar[dr]^-1 & \\ 
 \ker(d_0)\otimes \Delta^1 \ar[r]^-h & \ker(d_0) \\ 
 \ker(d_0) \ar[u] \ar[ur]_-{\mathrm{id}} }
 \]
in $s\cK_{/B}$.  Applying the fiberwise geometric realization functor $|\cdot |$, and using the fact that 
$|\cdot |\colon s\cK_{/B}\to \cK_{/B}$ is compatible with the geometric realization functor $|\cdot |\colon \sSet\to \cK$ in the sense 
that there is an isomorphism $|X\otimes K|\cong |X|\otimes |K|$ for all objects $X$ of $s\cK_{/B}$ and simplicial sets $K$, 
we see that $h$ induces a fiberwise contraction of $\ker(d_0)$ onto $B$.   
\end{proof}

Write $p$ for the homomorphism $\Dec\, G\to G$.  From the lemma we see that $p_n\colon \Dec_nG\to G_n$ is 
split for all $n\geq 0$ and hence the induced homomorphism $|\Dec_n G|\to G_n$ is also split.  
Therefore Corollary~\ref{corr:surj homo of grps} implies that $|\Dec\, G|\to G$ is a local fibration 
in $s\cK_{/B}$.  The following lemma is an immediate consequence of this, using 
the well known fact that $\Wbar\colon s\Gp\to s\Set$ sends surjective Kan fibrations to Kan fibrations.  

\begin{lemma} 
The induced map $p\colon \overline{W}|\Dec\, G|\to \overline{W}G$ is a local fibration of simplicial presheaves on $B$.  
\end{lemma} 

Since $p$ is a local 
fibration with fiber $\overline{W}|\ker(p)|$ we have (see \cite{Brown}) an exact sequence of sets 
\begin{equation} 
\label{W bar exact} 
[1,\overline{W} |\ker(p)|]\to [1,\overline{W} |\Dec\, G| ] \to [1,\overline{W} G] 
\end{equation}
Recall from Lemma~\ref{summary of dec} above that for every $n\geq 0$ the group 
$|\ker(p_n)|$ in $\cK_{/B}$ is fiberwise contractible, fibrant and well sectioned.  
Therefore, from the exactness of the sequence~\eqref{W bar exact}, we see that 
to prove the map $\W |\Dec\, G|\to \W G$ induces an isomorphism 
$[1,\W |\Dec\, G| ] \to [1,\W G]$ it is enough to prove the following lemmas.  

\begin{lemma} 
\label{lem:surjective}
Let $G$ be a group in $s\cK_{/B}$.  Then the map 
\[
[1,\W |\Dec\, G|]\to [1,\W G]  
\]
in $\Ho(s\Pre(B))$ 
induced by 
the map $\W |\Dec\, G|\to \W G$ in $s\cK_{/B}$ 
is surjective. 
\end{lemma} 

\begin{lemma} 
\label{lem:contractible}
Let $G$ be a well sectioned, fibrant group in $s\cK_{/B}$ such that $G_n$ is fiberwise contractible for all $n\geq 0$.  
Suppose that $B$ is paracompact and Hausdorff,  then $[1,\W G]$ is trivial.   
\end{lemma} 

\begin{proof}[Proof of Lemma~\ref{lem:surjective}]
We have seen above that the map $|\Dec\, G|\to G$ is degree-wise a split homomorphism 
of groups.  It follows that the corresponding homomorphism of groups $|\Dec\, G|\to G$ in $s\Pre(B)$ 
is object-wise surjective, in the sense that for each open set $U\subset B$, the homomorphism of 
simplicial groups $|\Dec\, G|(U)\to G(U)$ is surjective.  
Therefore $\Wbar |\Dec\, G|\to \Wbar G$ is a fibration in the projective model structure on $s\Pre(B)$ 
and the result follows from Lemma~\ref{surjectivity}. 
\end{proof}

\begin{proof}[Proof of Lemma~\ref{lem:contractible}] 
From Proposition~\ref{hypercover} above we see that for any 
well sectioned, fibrant parametrized simplicial group $G$ there is an isomorphism 
\[ 
[1,\overline{W} |\Dec\, G| ] \cong [1,\overline{W} |G|] 
\] 
and we have just seen that there is a surjection 
\begin{equation}
\label{surjection}
[1,\overline{W} |\Dec\, G| ] \twoheadrightarrow [1,\overline{W} G].   
\end{equation}
Consider the set $[1,\overline{W} |G|]$.  Earlier we have recalled how this is shown in 
\cite{Jardine3} to be isomorphic to the \v{C}ech cohomology set 
$\check{H}^1(B,|G|)$.  Since $G_n$ is fiberwise contractible for all $n\geq 0$, $G$ is good 
and the constant simplicial object $B$ is good, it follows that 
$|G|$ is fiberwise contractible.  We claim that every principal $|G|$ bundle 
on the paracompact space $B$ is trivial and so $[1,\overline{W}|G|]$ is trivial.  
It then follows that  
$[1,\overline{W} |\Dec\, G| ]$ is trivial; since the map~\eqref{surjection}   
is surjective it follows that $[1,\overline{W} G]$ is trivial as required.  
It remains to prove the claim.  It is enough to prove that any 
principal $|G|$-bundle $P$ on $B$ has a section under the hypotheses above.  
Since $|G|$ is fiberwise 
contractible, it follows that $P\to B$ is locally a fiber homotopy equivalence.  
Since $B$ is paracompact, it follows that $P\to B$ is a fiber homotopy equivalence, 
and hence has a section (see Corollary 3.2 of \cite{Dold}; the fact that Dold 
works in $\Top$ rather than $\cK$ presents no difficulty).    
\end{proof}
 
\section*{Acknowledgements} 

I foremost would like to thank Urs Schreiber for his encouragement, and for 
many helpful email discussions as well as for his comments 
on earlier drafts of these notes.   I would also 
like to thank Tom Leinster and Thomas Nikolaus for useful conversations, and 
Jim Stasheff for his comments on an initial draft of this paper.  Finally, I would like to thank 
John Baez for introducing me to simplicial homotopy theory.

\end{document}